\documentclass[letterpaper]{article}

\usepackage{charter}
\usepackage[utf8]{inputenc}
\usepackage[T1]{fontenc}
\usepackage[dvipsnames]{xcolor}
\usepackage{hyperref}
\usepackage{enumitem}
\usepackage[super]{nth}
\usepackage[gen]{eurosym}
\usepackage{graphicx}
\usepackage{amssymb}
\usepackage{amsmath}
\usepackage{stmaryrd}
\usepackage{algorithm}
\usepackage{algorithmic}
\usepackage{dsfont}
\usepackage{mathtools}
\usepackage{tikz}
\usepackage[round]{natbib}
\usepackage{apxproof}
\usepackage{booktabs}
\usepackage{amsfonts}
\usepackage{amsthm}
\usepackage[font=small]{caption}
\usepackage{parskip}

\renewcommand{\leq}{\leqslant}
\renewcommand{\geq}{\geqslant}
\renewcommand{\epsilon}{\varepsilon}

\DeclareMathOperator*{\argmin}{\mathrm{arg\,min}}

\DeclarePairedDelimiter\floor{\lfloor}{\rfloor}

\newtheorem{theorem}{Theorem}[section]

\newtheorem{fact}[theorem]{Fact}
\newtheorem{lemma}[theorem]{Lemma}

\newtheorem{remark}[theorem]{Remark}

\makeatletter
\newenvironment{template}[1][htb]{%
    \renewcommand{\ALG@name}{Template}
   \begin{algorithm}[#1]%
  }{\end{algorithm}}
\makeatother

\hypersetup{
    colorlinks=true,
    linkcolor={red!60!black},
    citecolor={green!40!black},
    urlcolor={blue!90!black}
}

\begin{document}

\begin{center}
 {\Large Projection-Free Adaptive Gradients for Large-Scale Optimization}
\end{center}

\vspace{7mm}

\noindent\textbf{Cyrille W.~Combettes}\hfill\href{mailto:cyrille@gatech.edu}{\ttfamily cyrille@gatech.edu}\\
\emph{\small School of Industrial and Systems Engineering and Department for AI in Society, Science, and Technology\\
Georgia Institute of Technology and Zuse Institute Berlin\\
Atlanta, GA, USA and Berlin, Germany}\\
\\
\textbf{Christoph Spiegel}\hfill\href{mailto:spiegel@zib.de}{\ttfamily spiegel@zib.de}\\
\emph{\small Department for AI in Society, Science, and Technology\\
Zuse Institute Berlin\\
Berlin, Germany}\\
\\
\textbf{Sebastian Pokutta}\hfill\href{mailto:pokutta@zib.de}{\ttfamily pokutta@zib.de}\\
\emph{\small Institute of Mathematics and Department for AI in Society, Science, and Technology\\
Technische Universit\"at Berlin and Zuse Institute Berlin\\
Berlin, Germany}

\vspace{7mm}

\begin{center}
\begin{minipage}{0.85\textwidth}
\begin{center}
 \textbf{Abstract}
\end{center}
\vspace{3mm}
 {\small The complexity in large-scale optimization can lie in both handling the objective function and handling the constraint set. In this respect, stochastic Frank-Wolfe algorithms occupy a unique position as they alleviate both computational burdens, by querying only approximate first-order information from the objective and by maintaining feasibility of the iterates without using projections. In this paper, we improve the quality of their first-order information by blending in adaptive gradients. We derive convergence rates and demonstrate the computational advantage of our method over the state-of-the-art stochastic Frank-Wolfe algorithms on both convex and nonconvex objectives. The experiments further show that our method can improve the performance of adaptive gradient algorithms for constrained optimization.}
\end{minipage}
\end{center}

\vspace{3mm}

\section{Introduction}

Consider the constrained finite-sum optimization problem
\begin{align}
 \min_{x\in\mathcal{C}}\left\{f(x)\coloneqq\frac{1}{m}\sum_{i=1}^mf_i(x)\right\},\label{pb}
\end{align}
where $\mathcal{C}\subset\mathbb{R}^n$ is a compact convex set and $f_1,\ldots,f_m\colon\mathbb{R}^n\rightarrow\mathbb{R}$ are smooth convex functions. We are interested in the case where the problem is \emph{large-scale}. Thus, we assume that exact function and gradient evaluations of the objective $f$ are excessively expensive and stochastic evaluations are significantly more efficient. We also assume that projections onto $\mathcal{C}$ are significantly more expensive than linear optimizations over $\mathcal{C}$. Such situations are encountered in, e.g., regression and classification problems over $\ell_p$-balls, nuclear norm-balls, or structured polytopes \citep{combettes21}. In this setting, stochastic Frank-Wolfe algorithms are the method of choice.

The Frank-Wolfe algorithm (FW) \citep{fw56}, a.k.a.~conditional gradient algorithm \citep{polyak66cg}, is a simple projection-free first-order algorithm for constrained minimization. At each iteration, it calls a linear minimization oracle $v_t\leftarrow\argmin_{v\in\mathcal{C}}\langle\nabla f(x_t),v\rangle$ given by the current gradient $\nabla f(x_t)$ and moves in the direction of $v_t$ by convex combination, updating the new iterate $x_{t+1}\leftarrow x_t+\gamma_t(v_t-x_t)$ where $\gamma_t\in\left[0,1\right]$. Hence, $x_{t+1}\in\mathcal{C}$ by convexity of $\mathcal{C}$ and there is no need to ensure feasibility via projections back onto $\mathcal{C}$. In short, FW avoids the projection step of gradient descent by moving in the direction of a point $v_t$ minimizing over $\mathcal{C}$ the linear approximation of $f$ at $x_t$. The Stochastic Frank-Wolfe algorithm (SFW) simply replaces the gradient $\nabla f(x_t)$ in the input to the linear minimization oracle with a stochastic estimator $\tilde{\nabla}f(x_t)\leftarrow(1/b_t)\sum_{i=i_1}^{i_{b_t}}\nabla f_i(x_t)$, where $i_1,\ldots,i_{b_t}$ are sampled i.i.d.~uniformly at random from $\llbracket1,m\rrbracket$. When the batch-sizes scale as $b_t=\Theta(t^2)$, SFW converges with rate $\mathcal{O}(1/t)$ \citep{hazan16}.

Many variants have been proposed to improve the practical efficiency of SFW, also converging at a rate of $\mathcal{O}(1/t)$. From a theoretical standpoint, a popular measure of efficiency in large-scale optimization is the number of gradient evaluations required to achieve $\epsilon$-convergence. To this end, the Stochastic Variance-Reduced Frank-Wolfe algorithm (SVRF) \citep{hazan16} integrates variance reduction \citep{svrg13,zhang13} in the estimation of the stochastic gradients to improve the batch-size rate to $b_t=\Theta(t)$. The STOchastic variance-Reduced  Conditional gradient sliding algorithm (STORC) \citep{hazan16} builds on the Conditional Gradient Sliding algorithm \citep{lan16cgs} and further reduces the total number of gradient evaluations by half an order of magnitude, although STORC is not as competitive as SVRF in practice \citep[Section~5]{hazan16}. This may be because SVRF obtains more progress per gradient evaluation or because the analysis of STORC is more precise. When the objective is additively separable in the data samples, \citet{negiar20} present a Constant batch-size Stochastic Frank-Wolfe algorithm (CSFW) where the batch-sizes do not need to grow over time, i.e., $b_t=\Theta(1)$; in practice, they set $b_t\leftarrow\floor{m/100}$.

Hence, the number of gradient evaluations required to achieve convergence, although appealing for its theoretical insight, may not necessarily reflect the relative performances of different algorithms in practice. Furthermore, focusing solely on reducing this quantity may actually be ineffective \citep{defazio19}. In this paper, we take a different route and leverage recent advances in optimization to improve the performance of SFW (and variants) via a better use of first-order information. To the best of our knowledge, it has not yet been explored how to take advantage of adaptive gradients \citep{adagrad11,mcmahan10}, which have been very successful in modern large-scale learning (see, e.g., \citet{dean12}), in projection-free optimization. Adaptive gradient algorithms consist in setting entry-wise step-sizes based on first-order information from past iterates. An interpretation of the success of these methods is that they provide a feature-specific learning rate, which is particularly useful when informative features from the dataset are present in the form of rare events. From an optimization standpoint, these adaptive step-sizes fit better to the loss landscape and alleviate the struggle with ill-conditioning, without requiring access to second-order information.

\paragraph{Contributions.} We show that the blend of adaptive gradients and Frank-Wolfe algorithms is successful. We propose a generic template for improving the performance of stochastic Frank-Wolfe algorithms, which applies to all the aforementioned variants. Our method consists in solving the non-Euclidean projection subproblems occurring in the adaptive gradient algorithms \emph{very} incompletely via a fixed and small number of $K$ iterations of the Frank-Wolfe algorithm. We establish convergence guarantees for different implementations of our method, and we demonstrate its computational advantage over the state-of-the-art stochastic Frank-Wolfe algorithms on both convex and nonconvex objectives. Furthermore, the experiments also show that our method can improve the performance of adaptive gradient algorithms on constrained optimization problems. While adaptive gradient algorithms require a non-Euclidean projection at each iteration, our method is projection-free and still leverages adaptive gradients.

\paragraph{Outline.} We start with background materials on stochastic Frank-Wolfe algorithms and adaptive gradient algorithms in Section~\ref{sec:prelim}. In Section~\ref{sec:adafw}, we motivate our approach and present our method through a generic template. We further propose specific implementations and analyze their respective convergence properties. We end that section with practical recommendations and we report computational experiments in Section~\ref{sec:exp}. We conclude the paper with some final remarks in Section~\ref{sec:final}. All proofs can be found in Appendix~\ref{apx:proofs}. Appendix~\ref{apx:K} contains an analysis of the sensitivity to $K$.

\section{Preliminaries}
\label{sec:prelim}

\subsection{Notation and definitions}
\label{sec:nota}

We work in the Euclidean space $(\mathbb{R}^n,\langle\cdot,\cdot\rangle)$ equipped with the standard inner product. We consider problem~\eqref{pb} where $\mathcal{C}\subset\mathbb{R}^n$ is a compact convex set and $f_1,\ldots,f_m\colon\mathbb{R}^n\rightarrow\mathbb{R}$ are smooth functions. Unless declared otherwise, we will also assume that they are convex. The objective function is $f\coloneqq(1/m)\sum_{i=1}^mf_i$. For every $i,j\in\mathbb{N}$ such that $i\leq j$, the brackets $\llbracket i,j\rrbracket$ denote the set of integers between (and including) $i$ and $j$.
Let $L\coloneqq\max_{i\in\llbracket1,m\rrbracket}\max_{x,y\in\mathcal{C}}\|\nabla f_i(y)-\nabla f_i(x)\|_2$ and $G\coloneqq\max_{i\in\llbracket1,m\rrbracket}\max_{x\in\mathcal{C}}\|\nabla f_i(x)\|_2$. In particular, for all $x,y\in\mathcal{C}$, we have
\begin{align*}
 f(y)
 \leq f(x)+\langle\nabla f(x),y-x\rangle+\frac{L}{2}\|y-x\|_2^2.
\end{align*}

For all $x\in\mathbb{R}^n$ and $i\in\llbracket1,n\rrbracket$, $[x]_i$ denotes the $i$-th entry of $x$. We write $\operatorname{diag}(x)$ for the diagonal matrix $([x]_i\mathds{1}_{\{i=j\}})_{(i,j)\in\llbracket1,n\rrbracket\times\llbracket1,n\rrbracket}\in\mathbb{R}^{n\times n}$.
For all $p\in\left[1,+\infty\right]$, let $D_p\coloneqq\max_{x,y\in\mathcal{C}}\|y-x\|_p$ and $\mathcal{B}_p(\tau)\coloneqq\{x\in\mathbb{R}^n\mid\|x\|_p\leq\tau\}$. In the case of the $\ell_2$-norm, we simply write $D$ when the context is suitable. For all $A\in\mathbb{R}^{m\times n}$, let $D_p^A\coloneqq\max_{x,y\in\mathcal{C}}\|A(y-x)\|_p$ and $\|A\|_p\coloneqq\max_{\|x\|_p=1}\|Ax\|_p$. Given a symmetric matrix $H\in\mathbb{R}^{n\times n}$, we denote by $\lambda_{\min}(H)$ and $\lambda_{\max}(H)$ its smallest and largest eigenvalues. When $H$ is also positive definite, then $(x,y)\in\mathbb{R}^n\times\mathbb{R}^n\mapsto\langle x, Hy\rangle$ is an inner product with associated norm $\|\cdot\|_H\colon x\in\mathbb{R}^n\mapsto\sqrt{\langle x,Hx\rangle}$, and it holds
\begin{align}
 \lambda_{\min}(H)\|\cdot\|_2^2
 \leq\|\cdot\|_H^2
 \leq\lambda_{\max}(H)\|\cdot\|_2^2.\label{h2}
\end{align}
Lastly, we denote by $\mathsf{E}$ the expectation with respect to all the randomness in the system under study.

\subsection{Stochastic Frank-Wolfe algorithms}
\label{sec:sfw}

We present in Template~\ref{fw} the generic template for stochastic Frank-Wolfe algorithms.

\vspace{1mm}
\begin{template}[h]
\caption{Stochastic Frank-Wolfe}
\label{fw}
\begin{algorithmic}[1]
\REQUIRE Start point $x_0\in\mathcal{C}$, step-sizes $\gamma_t\in\left[0,1\right]$.
\FOR{$t=0$ \textbf{to} $T-1$}
\STATE Update the gradient estimator $\tilde{\nabla}f(x_t)$\label{tilde}
\STATE$v_t\leftarrow\argmin\limits_{v\in\mathcal{C}}\langle\tilde{\nabla}f(x_t),v\rangle$\label{lmo}
\STATE$x_{t+1}\leftarrow x_t+\gamma_t(v_t-x_t)$\label{new}
\ENDFOR
\end{algorithmic}
\end{template}
\vspace{1mm}

When $\tilde{\nabla}f(x_t)\leftarrow\nabla f(x_t)$, we obtain the original Frank-Wolfe algorithm (FW) \citep{fw56}, a.k.a.~conditional gradient algorithm \citep{polyak66cg}. The update by convex combination in Line~\ref{new} ensures that $x_{t+1}\in\mathcal{C}$ and FW thereby does not need projections back onto $\mathcal{C}$. The point $v_t\in\mathcal{C}$ is chosen in Line~\ref{lmo} so that the sequence of iterates converges, as the iterates move in the directions $v_t-x_t$ satisfying a sufficient first-order condition key to the convergence analyses: let $x^*\in\argmin_\mathcal{C}f$, then
\begin{align*}
 \langle\nabla f(x_t),v_t-x_t\rangle
 &=\langle\tilde{\nabla}f(x_t),v_t-x_t\rangle+
 \langle\nabla f(x_t)-\tilde{\nabla}f(x_t),v_t-x_t\rangle\\
 &\leq\langle\tilde{\nabla}f(x_t),x^*-x_t\rangle+
 \langle\nabla f(x_t)-\tilde{\nabla}f(x_t),v_t-x_t\rangle\\
 &=\langle\nabla f(x_t),x^*-x_t\rangle+\langle\tilde{\nabla}f(x_t)-\nabla f(x_t),x^*-v_t\rangle\\
 &\leq-(f(x_t)-f(x^*))+
 \langle\tilde{\nabla}f(x_t)-\nabla f(x_t),x^*-v_t\rangle,
\end{align*}
by convexity of $f$. The second term can be controlled by properties of the gradient estimator $\tilde{\nabla}f(x_t)$, often by using first the Cauchy-Schwarz or H\"older's inequality.

\vspace{2mm}
\begin{table}[h]
\caption{Gradient estimator updates in stochastic Frank-Wolfe algorithms. The indices $i_1,\ldots,i_{b_t}$ are sampled i.i.d.~uniformly at random from $\llbracket1,m\rrbracket$.}
\label{table:tilde}
\centering{
\begin{tabular}{lll}
\toprule
\textbf{Algorithm}&\textbf{Update $\tilde{\nabla}f(x_t)$ in Line~\ref{tilde}}&\textbf{Additional information}\\
\midrule
SFW&$\displaystyle\frac{1}{b_t}\sum_{i=i_1}^{i_{b_t}}\nabla f_i(x_t)$&$\varnothing$\\
SVRF&$\displaystyle\nabla f(\tilde{x}_t)+\frac{1}{b_t}\sum_{i=i_1}^{i_{b_t}}(\nabla f_i(x_t)-\nabla f_i(\tilde{x}_t))$&$\tilde{x}_t$ is the last snapshot iterate\\
SPIDER-FW&$\displaystyle\nabla f(\tilde{x}_t)+\frac{1}{b_t}\sum_{i=i_1}^{i_{b_t}}(\nabla f_i(x_t)-\nabla f_i(x_{t-1}))$&$\tilde{x}_t$ is the last snapshot iterate\\
ORGFW&$\displaystyle\frac{1}{b_t}\sum_{i=i_1}^{i_{b_t}}\nabla f_i(x_t)+(1-\rho_t)\left(\tilde{\nabla}f(x_{t-1})-\frac{1}{b_t}\sum_{i=i_1}^{i_{b_t}}\nabla f_i(x_{t-1})\right)$&$\rho_t$ is the momentum parameter\\
CSFW&$\displaystyle\tilde{\nabla}f(x_{t-1})+\sum_{i=i_1}^{i_{b_t}}\left(\frac{1}{m}f_i'(\langle a_i,x_t\rangle)-[\alpha_{t-1}]_i\right)a_i$&Assumes separability of $f$ as\\
&and $[\alpha_t]_i\leftarrow
\begin{cases}
(1/m)f_i'(\langle a_i,x_t\rangle) &\text{if }i\in\{i_1,\ldots,i_{b_t}\}\\
[\alpha_{t-1}]_i&\text{else}
\end{cases}
$&$\displaystyle f(x)=\frac{1}{m}\sum_{i=1}^mf_i(\langle a_i,x\rangle)$\\
\bottomrule
\end{tabular}
}
\end{table}
\vspace{1mm}

In order to fit FW to the large-scale finite-sum setting of problem~\eqref{pb}, many \emph{stochastic} Frank-Wolfe algorithms have been developed. Most of them follow Template~\ref{fw} and differ only in how they update the gradient estimator $\tilde{\nabla}f(x_t)$ (Line~\ref{tilde}). In Table~\ref{table:tilde}, we report the strategies adopted in the (vanilla) Stochastic Frank-Wolfe algorithm (SFW) \citep{hazan16}, the Stochastic Variance-Reduced Frank-Wolfe algorithm (SVRF) \citep{hazan16}, the Stochastic Path-Integrated Differential EstimatoR Frank-Wolfe algorithm (SPIDER-FW) \citep{yurtsever19,shen19}, the Online stochastic Recursive Gradient-based Frank-Wolfe algorithm (ORGFW) \citep{orgfw20}, and the Constant batch-size Stochastic Frank-Wolfe algorithm (CSFW) \citep{negiar20}. SFW is the natural extension of FW to the large-scale setting of problem~\eqref{pb}, SVRF and SPIDER-FW integrate variance reduction based on the works of \citet{svrg13} and \citet{fang18} respectively, ORGFW uses a form of momentum inspired by \citet{cutkosky19}, and CSFW takes advantage of the additive separability of the objective function in the data samples, when applicable, following the design of \citet{schmidt17}. 

\subsection{The Adaptive Gradient algorithm}
\label{sec:adagrad}

The Adaptive Gradient algorithm (AdaGrad) \citep{adagrad11} (see also \citet{mcmahan10}) is presented in Algorithm~\ref{adagrad}.

\vspace{1mm}
\begin{algorithm}[h]
\caption{Adaptive Gradient (AdaGrad)}
\label{adagrad}
\begin{algorithmic}[1]
\REQUIRE Start point $x_0\in\mathcal{C}$, offset $\delta>0$, learning rate $\eta>0$.
\FOR{$t=0$ \textbf{to} $T-1$}
\STATE Update the gradient estimator $\tilde{\nabla}f(x_t)$
\STATE$H_t\leftarrow\operatorname{diag}\left(\delta1+\sqrt{\displaystyle\sum_{s=0}^t\tilde{\nabla}f(x_s)^2}\,\right)$\label{diag}
\STATE\label{new2}$x_{t+1}\leftarrow\argmin\limits_{x\in\mathcal{C}}\eta\langle\tilde{\nabla}f(x_t),x\rangle+\displaystyle\frac{1}{2}\|x-x_t\|_{H_t}^2$
\ENDFOR
\end{algorithmic}
\end{algorithm}
\vspace{1mm}

All operations in Line~\ref{diag} are entry-wise in $\mathbb{R}^n$. The matrix $H_t\in\mathbb{R}^{n\times n}$ is diagonal and satisfies for all $i,j\in\llbracket1,n\rrbracket$,
\begin{align}
 [H_t]_{i,j}=
 \begin{cases}
  \delta+\sqrt{\displaystyle\sum_{s=0}^t[\tilde{\nabla}f(x_s)]_i^2}&\text{if }i=j,\\
  0&\text{if }i\neq j.
 \end{cases}\label{h:adagrad}
\end{align}
The default value for the offset hyperparameter is $\delta\leftarrow10^{-8}$. The new iterate $x_{t+1}$ is computed in Line~\ref{new2} by solving a constrained convex quadratic minimization subproblem. By completing the square, this subproblem is equivalent to a non-Euclidean projection in the metric $\|\cdot\|_{H_t}$:
\begin{align}
 x_{t+1}
 \leftarrow\argmin_{x\in\mathcal{C}}\|x-(x_t-\eta H_t^{-1}\tilde{\nabla}f(x_t))\|_{H_t}.\label{proj}
\end{align}
Ignoring the constraint set $\mathcal{C}$ for ease of exposition, we obtain
\begin{align*}
 x_{t+1}
 \leftarrow x_t-\eta H_t^{-1}\tilde{\nabla}f(x_t),
\end{align*}
i.e., for every feature $i\in\llbracket1,n\rrbracket$,
\begin{align}
 [x_{t+1}]_i
 \leftarrow[x_t]_i-\frac{\eta[\tilde{\nabla}f(x_t)]_i}{\delta+\sqrt{\sum_{s=0}^t[\tilde{\nabla}f(x_s)]_i^2}}.\label{i}
\end{align}
Thus, the offset $\delta$ prevents from dividing by zero, and we can see that the step-size automatically scales with the geometry of the problem. In particular, infrequent features receive large step-sizes whenever they appear, allowing the algorithm to notice these rare but potentially very informative features.

The family of adaptive gradient algorithms originated with AdaGrad and expanded with RMSProp \citep{rmsprop12}, AdaDelta \citep{adadelta12}, Adam \citep{adam15}, AMSGrad \citep{reddi18}, and, e.g., AdaBound \citep{adabound19,keskar17}, with each new variant addressing some flaws in the previous ones: vanishing step-sizes, incomplete theory, generalization performance \citep{wilson17}, etc. For example, RMSProp uses an exponential moving average instead of a sum in Line~\ref{diag} in order to avoid vanishing step-sizes, since the sum in the denominator of~\eqref{i} can grow too fast for features with dense gradients.

\section{Frank-Wolfe with adaptive gradients}
\label{sec:adafw}

\subsection{Our approach}
\label{sec:approach}

When minimizing an objective over a constraint set, each iteration of AdaGrad can be relatively expensive as it needs to solve the subproblem
\begin{align}
 \min_{x\in\mathcal{C}}\eta\langle\tilde{\nabla}f(x_t),x\rangle+\displaystyle\frac{1}{2}\|x-x_t\|_{H_t}^2,\label{sub}
\end{align}
given in Line~\ref{new2}. By~\eqref{proj}, this is equivalent to a non-Euclidean projection of the unconstrained step $x_t-\eta H_t^{-1}\tilde{\nabla}f(x_t)$. Thus, we could reduce the complexity of AdaGrad by avoiding this projection and moving in the direction of $\argmin_{v\in\mathcal{C}}\langle G_t,v\rangle$, where $-G_t=-H_t^{-1}\tilde{\nabla}f(x_t)$ denotes the unconstrained descent direction of AdaGrad, as was done in FW for gradient descent where the descent direction is $-G_t=-\nabla f(x_t)$. However, by doing so we may lose the precious properties of the descent directions of AdaGrad, as the directions returned by $\argmin_{v\in\mathcal{C}}\langle G_t,v\rangle$ can be significantly different from $-G_t$ \citep{combettes20}; see \citet[Figure~32]{polyak87} for an early illustration of the phenomenon.

Thus, instead of avoiding the subproblem~\eqref{sub}, we can consider solving it incompletely and via a projection-free algorithm. Following \citet{lan16cgs}, at each iteration we could use FW to solve~\eqref{sub} until some specified accuracy $\phi_t$ is reached, which we check via the duality gap $\max_{v\in\mathcal{C}}\langle\eta\tilde{\nabla}f(x_t)+H_t(x-x_t),v\rangle$ \citep{jaggi13fw} (Fact~\ref{fact:wolfe}). The solution to this procedure would then constitute the new iterate $x_{t+1}$. The subproblem~\eqref{sub} is easy to address since the objective is a simple convex quadratic function, so we can evaluate its exact gradient cheaply and derive the optimal step-size in any descent direction.

However, in order to provide nice theoretical analyses, the sequence of accuracies $(\phi_t)_{t\in\llbracket0,T-1\rrbracket}$ would need to decay to zero relatively fast, which means that we are back to solving the subproblems completely. This is very time-consuming and overkill in practice. Therefore, instead we propose to perform a fixed number of $K$ iterations on the subproblems, where $K$ is chosen small, e.g., $K\sim5$. Hence, we choose to leverage just a small amount of information from the adaptive metric $H_t$, and claim that this will be enough in practice.

\subsection{The algorithm}
\label{sec:algo}

We present our method via a generic template in Template~\ref{adafw}. The matrix $H_t$ is allowed to be any diagonal matrix with positive entries. Hence, we can apply the AdaGrad strategy~\eqref{h:adagrad}, but we can also apply any other variant. In our convergence analyses, we need to ensure that $\sup_{t\in\mathbb{N}}\lambda_{\max}(H_t)<+\infty$, so we propose to clip the entries of $H_t$, as done in \citet{adabound19}. Lines~\ref{sub:start}--\ref{sub:end} apply $K$ iterations of FW on
\begin{align}
 \min_{x\in\mathcal{C}}\left\{Q_t(x)\coloneqq f(x_t)+\langle\tilde{\nabla}f(x_t),x-x_t\rangle+\frac{1}{2\eta_t}\|x-x_t\|_{H_t}^2\right\}.\label{sub2}
\end{align}
This is exactly the subproblem~\eqref{sub} with a time-varying learning rate $\eta_t>0$. We denote by $y_k^{(t)}$ for $k\in\llbracket0,K\rrbracket$ the iterates on the subproblem~\eqref{sub2}, starting from $y_0^{(t)}\leftarrow x_t$ (Line~\ref{sub:start}) and ending at $x_{t+1}\leftarrow y_K^{(t)}$ (Line~\ref{new3}). The step-size $\gamma_k^{(t)}$ in Line~\ref{algo:gammak} is optimal in the sense that $\gamma_k^{(t)}=\argmin_{\gamma\in\left[0,\gamma_t\right]}Q_t(y_k^{(t)}+\gamma(v_k^{(t)}-y_k^{(t)}))$ (Lemma~\ref{lem:sub}), where the upper bound $\gamma_t$ ensures convergence of the sequence $(x_t)_{t\in\llbracket0,T\rrbracket}$. 

\vspace{1mm}
\begin{template}[h]
\caption{Frank-Wolfe with adaptive gradients}
\label{adafw}
\begin{algorithmic}[1]
\REQUIRE Start point $x_0\in\mathcal{C}$, bounds $0<\lambda_t^-\leq\lambda_{t+1}^-\leq\lambda_{t+1}^+\leq\lambda_t^+$, number of inner iterations $K\in\mathbb{N}\backslash\{0\}$, learning rates $\eta_t>0$, step-size bounds $\gamma_t\in\left[0,1\right]$.
\FOR{$t=0$ \textbf{to} $T-1$}
\STATE Update the gradient estimator $\tilde{\nabla}f(x_t)$\label{adafw:g}
\STATE Update the diagonal matrix $H_t$ and clip its entries to $[\lambda_t^-,\lambda_t^+]$\label{h}
\STATE$y_0^{(t)}\leftarrow x_t$\label{sub:start}
\FOR{$k=0$ \textbf{to} $K-1$}
\STATE$\nabla Q_t(y_k^{(t)})\leftarrow\tilde{\nabla}f(x_t)+\displaystyle\frac{1}{\eta_t}H_t(y_k^{(t)}-x_t)$\label{sub:Q}
\STATE$v_k^{(t)}\leftarrow\argmin\limits_{v\in\mathcal{C}}\langle\nabla Q_t(y_k^{(t)}),v\rangle$
\STATE$\gamma_k^{(t)}\leftarrow\min\left\{\eta_t\displaystyle\frac{\langle\nabla Q_t(y_k^{(t)}),y_k^{(t)}-v_k^{(t)}\rangle}{\|y_k^{(t)}-v_k^{(t)}\|_{H_t}^2},\gamma_t\right\}$\label{algo:gammak}
\STATE$y_{k+1}\leftarrow y_k^{(t)}+\gamma_k^{(t)}(v_k^{(t)}-y_k^{(t)})$
\ENDFOR\label{sub:end}
\STATE$x_{t+1}\leftarrow y_K^{(t)}$\label{new3}
\ENDFOR
\end{algorithmic}
\end{template}
\vspace{1mm}

In Sections~\ref{sec:adasfw}--\ref{sec:adacsfw}, we propose specific implementations of Template~\ref{adafw}, where gradients are estimated as done in SFW, SVRF, or CSFW (Table~\ref{table:tilde}). The diagonal matrices $H_t$, $t\in\llbracket0,T-1\rrbracket$, can still be very general. The derived algorithms are named AdaSFW, AdaSVRF, and AdaCSFW respectively, and we analyze their convergence rates. We follow the assumptions from Section~\ref{sec:nota}; in particular, unless declared otherwise, $f_1,\ldots,f_m$ are smooth convex functions. By clipping the entries of $H_t$, we can control the change rate of the adaptive gradients. However, allowing $H_t$ to be any diagonal matrix with positive entries comes at a price: the upper bound on the convergence rate of the objective is worse than that of the vanilla SFW for example. Nonetheless, the method converges significantly faster in practice (Section~\ref{sec:exp}).

\subsection{SFW with adaptive gradients}
\label{sec:adasfw}

We present AdaSFW in Algorithm~\ref{adasfw}. It simply estimates the gradient by averaging over a minibatch.

\vspace{1mm}
\begin{algorithm}[h]
\caption{AdaSFW}
\label{adasfw}
\begin{algorithmic}[1]
\REQUIRE Start point $x_0\in\mathcal{C}$, batch-sizes $b_t\in\mathbb{N}\backslash\{0\}$.
\FOR{$t=0$ \textbf{to} $T-1$}
\STATE$i_1,\ldots,i_{b_t}\overset{\text{i.i.d.}}{\sim}\mathcal{U}(\llbracket1,m\rrbracket)$
\STATE$\tilde{\nabla}f(x_t)\leftarrow\displaystyle\frac{1}{b_t}\sum_{i=i_1}^{i_{b_t}}\nabla f_i(x_t)$
\STATE Execute Lines~\ref{h}--\ref{new3} of Template~\ref{adafw}
\ENDFOR
\end{algorithmic}
\end{algorithm}
\vspace{1mm}

\begin{theorem}
 \label{th:adasfw}
 Consider AdaSFW (Algorithm~\ref{adasfw}) with $b_t\leftarrow\big(G(t+2)/(LD)\big)^2$, $\eta_t\leftarrow\lambda_t^-/L$, and $\gamma_t\leftarrow2/(t+2)$, and let $\kappa\coloneqq\lambda_0^+/\lambda_0^-$. Then for all $t\in\llbracket1,T\rrbracket$,
 \begin{align*}
  \mathsf{E}[f(x_t)]-\min_\mathcal{C}f
  \leq\frac{2LD^2(K+1+\kappa)}{t+1}.
 \end{align*}
\end{theorem}

\begin{remark}
 \label{rem:sfw}
 Theorem~\ref{th:adasfw} simply states that we need to scale the batch-sizes as $b_t=\Theta(t^2)$. We do not need to search for the values of $G$, $L$, or $D$ in practice. The same holds for SFW \citep{hazan16}.
\end{remark}

We propose in Theorem~\ref{th:ncvx} a convergence analysis of AdaSFW on nonconvex objectives. We measure convergence via the duality gap $g\colon x\in\mathcal{C}\mapsto\max_{v\in\mathcal{C}}\langle\nabla f(x),x-v\rangle$ \citep{jaggi13fw} as done in, e.g., \citet{lj16,reddi16}. The duality gap satisfies $g(x)\geq0$, $g(x)=0$ if and only if $x$ is a stationary point, and, when $f$ is convex, $g(x)\geq f(x)-\min_\mathcal{C}f$ (Fact~\ref{fact:wolfe}).

\begin{theorem}
\label{th:ncvx}
 Suppose that $f_1,\ldots,f_m$ are not necessarily convex. Consider AdaSFW (Algorithm~\ref{adasfw}) with $b_t\leftarrow\big(G/(LD)\big)^2(t+1)$, $\eta_t\leftarrow\lambda_t^-/L$, and $\gamma_t\leftarrow1/(t+1)^{1/2+\nu}$ where $\nu\in\left]0,1/2\right[$, and let $\kappa\coloneqq\lambda_0^+/\lambda_0^-$. For all $t\in\llbracket0,T\rrbracket$, let $X_t$ be sampled uniformly at random from $\{x_0,\ldots,x_t\}$. Then for all $t\in\llbracket0,T-1\rrbracket$,
 \begin{align*}
  \mathsf{E}[g(X_t)]
  \leq\frac{(f(x_0)-\min_\mathcal{C}f)+LD^2(K+1+\kappa/2)S}{(t+1)^{1/2-\nu}},
 \end{align*}
 where $S\coloneqq\sum_{s=0}^{+\infty}1/(s+1)^{1+\nu}\in\mathbb{R}_+$. Alternatively, if the time horizon $T$ is fixed, then with $b_t\leftarrow(G/(LD))^2T$ and $\gamma_t\leftarrow1/\sqrt{T}$, 
 \begin{align*}
  \mathsf{E}[g(X_{T-1})]
  \leq\frac{(f(x_0)-\min_\mathcal{C}f)+LD^2(K+1+\kappa/2)}{\sqrt{T}}.
 \end{align*}
\end{theorem}

\begin{remark}
 In the first setting of Theorem~\ref{th:ncvx}, if $\nu\leftarrow0.05$ for example, then $\mathsf{E}[g(X_t)]=\mathcal{O}(1/t^{0.45})$ and $S\approx20.6$.
\end{remark}

\subsection{SVRF with adaptive gradients}
\label{sec:adasvrf}

We present AdaSVRF in Algorithm~\ref{adasvrf}. At every iteration $t=s_k$, $k\in\mathbb{N}$, it computes the exact gradient of the iterate, saves it into memory, then builds the gradient estimator $\tilde{\nabla}f(x_t)$ in the following iterations $t\in\llbracket s_k+1,s_{k+1}-1\rrbracket$ from this snapshot. Compared to AdaSFW, the variance $\mathsf{E}[\|\tilde{\nabla}f(x_t)-\nabla f(x_t)\|_2^2]$ of the estimator is effectively reduced. The snapshot iterate for $x_t$ is denoted by $\tilde{x}_t$.

\begin{theorem}
 \label{th:adasvrf}
 Consider AdaSVRF (Algorithm~\ref{adasvrf}) with $s_k\leftarrow2^k-1$, $b_t\leftarrow24(K+1+\kappa)(t+2)$ where $\kappa\coloneqq\lambda_0^+/\lambda_0^-$, $\eta_t\leftarrow\lambda_t^-/L$, and $\gamma_t\leftarrow2/(t+2)$. Then for all $t\in\llbracket1,T\rrbracket$,
 \begin{align*}
  \mathsf{E}[f(x_t)]-\min_\mathcal{C}f
  \leq\frac{2LD^2(K+1+\kappa)}{t+2}.
 \end{align*}
\end{theorem}

\vspace{1mm}
\begin{algorithm}[H]
\caption{AdaSVRF}
\label{adasvrf}
\begin{algorithmic}[1]
\REQUIRE Start point $x_0\in\mathcal{C}$, snapshot times $s_k<s_{k+1}$ with $s_0=0$, batch-sizes $b_t\in\mathbb{N}\backslash\{0\}$.
\FOR{$t=0$ \textbf{to} $T-1$}
\IF{$t\in\{s_k\mid k\in\mathbb{N}\}$}
\STATE$\tilde{x}_t\leftarrow x_t$
\STATE$\tilde{\nabla}f(x_t)\leftarrow\nabla f(\tilde{x}_t)$\label{adagradfw:shot}
\ELSE
\STATE$\tilde{x}_t\leftarrow\tilde{x}_{t-1}$
\STATE$i_1,\ldots,i_{b_t}\overset{\text{i.i.d.}}{\sim}\mathcal{U}(\llbracket1,m\rrbracket)$
\STATE$\tilde{\nabla}f(x_t)\leftarrow\nabla f(\tilde{x}_t)+\displaystyle\frac{1}{b_t}\sum_{i=i_1}^{i_{b_t}}(\nabla f_i(x_t)-\nabla f_i(\tilde{x}_t))$
\ENDIF
\STATE Execute Lines~\ref{h}--\ref{new3} of Template~\ref{adafw}
\ENDFOR
\end{algorithmic}
\end{algorithm}
\vspace{1mm}

\begin{remark}
 \label{rem:svrf}
 Consider the setting of Theorem~\ref{th:adasvrf} with the more general strategy $s_k\leftarrow2^{k+k_0}-2^{k_0}$ and $b_t\leftarrow8(2^{k_0+1}+1)(K+1+\kappa)(t+2)$ where $k_0\in\mathbb{N}$. Then the same result holds. Choosing $k_0>0$ is useful in practice to avoid computing exact gradients too many times in the early iterations.
\end{remark}

\subsection{CSFW with adaptive gradients}
\label{sec:adacsfw}

Here we assume the objective function to be additively separable in the data samples. The problem is
\begin{align*}
 \min_{x\in\mathcal{C}}\left\{f(x)\coloneqq\frac{1}{m}\sum_{i=1}^mf_i(\langle a_i,x\rangle)\right\},
\end{align*}
where $f_1,\ldots,f_m\colon\mathbb{R}\rightarrow\mathbb{R}$ are smooth convex functions and $a_1,\ldots,a_m\in\mathbb{R}^n$ are the data samples. Thus, $\nabla f(x)=(1/m)\sum_{i=1}^mf_i'(\langle a_i,x\rangle)a_i$. We present AdaCSFW in Algorithm~\ref{adacsfw}. It estimates the gradient with the quantity $\tilde{\nabla}f(x_t)=\sum_{i=1}^m[\alpha_t]_ia_i$ by iteratively updating entries of the vector $\alpha_t\in\mathbb{R}^m$.

\vspace{1mm}
\begin{algorithm}[h]
\caption{AdaCSFW}
\label{adacsfw}
\begin{algorithmic}[1]
\REQUIRE Start point $x_0\in\mathcal{C}$, batch-size $b\in\mathbb{N}\backslash\{0\}$.
\STATE$\alpha_{-1}\leftarrow0\in\mathbb{R}^m$
\STATE$\tilde{\nabla}f(x_{-1})\leftarrow0\in\mathbb{R}^n$
\FOR{$t=0$ \textbf{to} $T-1$}
\STATE$i_1,\ldots,i_b\overset{\text{i.i.d.}}{\sim}\mathcal{U}(\llbracket1,m\rrbracket)$
\FOR{$i=1$ \textbf{to} $m$}
\IF{$i\in\{i_1,\ldots,i_b\}$}
\STATE$[\alpha_t]_i\leftarrow\displaystyle\frac{1}{m}f_i'(\langle a_i,x_t\rangle)$
\ELSE
\STATE$[\alpha_t]_i\leftarrow[\alpha_{t-1}]_i$
\ENDIF
\ENDFOR
\STATE$\tilde{\nabla}f(x_t)\leftarrow\tilde{\nabla}f(x_{t-1})+\displaystyle\sum_{i=i_1}^{i_b}([\alpha_t]_i-[\alpha_{t-1}]_i)a_i$
\STATE Execute Lines~\ref{h}--\ref{new3} of Template~\ref{adafw}
\ENDFOR
\end{algorithmic}
\end{algorithm}
\vspace{1mm}

\begin{theorem}
 \label{th:adacsfw}
 Consider AdaCSFW (Algorithm~\ref{adacsfw}) with $\eta_t\leftarrow m\lambda_t^-/(L\|A\|_2^2)$ and $\gamma_t\leftarrow2/(t+2)$, and let $\kappa\coloneqq\lambda_0^+/\lambda_0^-$. Then for all $t\in\llbracket1,T\rrbracket$,
 \begin{align*}
  \mathsf{E}[f(x_t)]-\min_\mathcal{C}f
  &\leq\frac{2L}{t+1}\left(4K(K+1)D_1^AD_\infty^A\left(\frac{1}{b}-\frac{1}{m}\right)+\frac{\kappa\|A\|_2^2D_2^2}{m}\right)\\
  &\quad+\frac{2(K+1)D_\infty^A(m/b)^2}{t(t+1)}\left(\|f'(x_0)-\alpha_0\|_1+\frac{16KLD_1^A}{b}\right).
 \end{align*}
\end{theorem}

\subsection{Practical recommendations}
\label{sec:practice}

We end this section with some practical recommendations. Following Remark~\ref{rem:svrf}, for convex objectives (Section~\ref{sec:cvx}) we set $k_0\leftarrow4$ in SVRF and AdaSVRF; for nonconvex objectives (Section~\ref{sec:ncvx}), we took snapshots once per epoch. In all variants of Template~\ref{adafw}, we used the AdaGrad strategy~\eqref{h:adagrad} for the adaptive metric $H_t$. We did not need to clip its entries. The offset was set to the default value $\delta\leftarrow10^{-8}$. We picked a constant value for the learning rate $\eta_t$, tuned in the range $\{10^{i/2}\mid i\in\mathbb{Z}\}$ by starting from $\{10^{i/2}\mid i\in\{-2,0,2\}\}$ and then narrowing the search space to $\{10^{i/2}\mid i\in\{i_\text{best}-1,i_\text{best},i_\text{best}+1\}\}$, and extending it if the new $i_\text{best}$ is at an endpoint. We did not need to bound the step-sizes $\gamma_k^{(t)}$, i.e., we set $\gamma_t\leftarrow1$. Either way, we noticed that the bounds $\gamma_t$ obtained from the theoretical analyses were not active in our experiments. Lastly, we found $K\sim5$ to be a good default value in general, as it provides both low complexity and high performance. A sensitivity analysis is available in Appendix~\ref{apx:K}.

\section{Computational experiments}
\label{sec:exp}

In this section, we conduct a computational study of our proposed method. We compare it to the state-of-the-art stochastic Frank-Wolfe algorithms, SFW, SVRF \citep{hazan16}, SPIDER-FW \citep{yurtsever19,shen19}, ORGFW \citep{orgfw20}, and CSFW \citep{negiar20}, as well as to the adaptive gradient algorithms AdaGrad \citep{adagrad11,mcmahan10} and AMSGrad \citep{reddi18}, which are usually applied to unconstrained problems in the literature. We chose AMSGrad as it solves the non-convergence issue of Adam \citep{adam15}. The goal is to demonstrate:
\begin{enumerate}[label=(\roman*)]
 \item that our method improves the performance of projection-free algorithms by blending in adaptive gradients, and
 \item that it can also be viewed as an efficient adaptive gradient method for constrained optimization, by being projection-free.
\end{enumerate}

At each iteration, AdaGrad and AMSGrad both require a non-Euclidean projection onto the constraint set $\mathcal{C}$. When $\mathcal{C}$ is an $\ell_1$-ball, we compute the projection as proposed in \citet[Section~5.2]{adagrad11}\footnote{In \citet[Figure~3]{adagrad11}, the \texttt{if} statement in Line~2 should be on the condition ``$\sum_ia_iv_i\leq c$'' instead of ``$\sum_iv_i\leq c$''.}. For both algorithms, the learning rate $\eta$ is tuned as explained in Section~\ref{sec:practice}.

\subsection{Convex objectives}
\label{sec:cvx}

We compare the algorithms on three standard convex optimization problems. We apply Template~\ref{adafw} to the best performing variant, demonstrating its flexibility and consistent performance. The performance of each algorithm is evaluated via the duality gap $\max_{v\in\mathcal{C}}\langle\nabla f(x_t),x_t-v\rangle$. When $\min_\mathcal{C}f$ is unknown, the duality gap serves as a measure of convergence and as a stopping criterion (Fact~\ref{fact:wolfe}). For the batch-sizes, we follow the recommendations given by the theoretical analyses of the respective algorithms. In SFW and SVRF, by Remark~\ref{rem:sfw} we set $b_t\sim t^2/\sqrt{m}$ and $b_t\sim t$ respectively, making sure $b_t$ does not grow too fast and stays small compared to the full batch-size $m$. We have $b_t\leftarrow\max\{2^k\mid t+1\geq2^k, k\in\mathbb{N}\}$ in SPIDER-FW, and $b_t\leftarrow\floor{m/100}$ in ORGFW, CSFW, AdaGrad, and AMSGrad, following \citet{negiar20} for algorithms where the batch-sizes do not need to grow over time.

\paragraph{Support vector classification.}\label{para:svm} We start with a support vector classification experiment from \citet[Equation~(4.3.11)]{duchi18}. Since our work only deals with smooth objective functions\footnote{For Frank-Wolfe on nonsmooth objectives, see \citet{argyriou14}.}, we smoothen the hinge loss by taking its square, as done in, e.g., \citet{zhang01}. The problem is
\begin{align*}
 \min_{x\in\mathbb{R}^n}\;&\frac{1}{m}\sum_{i=1}^m\max\{0,1-y_i\langle a_i,x\rangle\}^2\\
 \text{s.t.}\;&\|x\|_\infty\leq\tau,
\end{align*}
where the data is generated as follows. For every $(i,j)\in\llbracket1,m\rrbracket\times\llbracket1,n\rrbracket$, let $a_{i,j}=0$ with probability $1-1/j$, else $a_{i,j}=\pm1$ equiprobably. Thus, the data matrix $A\in\mathbb{R}^{m\times n}$ has significant variability in the frequency of the features. Then let $u\sim\mathcal{U}(\{-1,1\}^n)$, and $y_i=\operatorname{sign}(\langle a_i,u\rangle)$ with probability $0.95$ else $y_i=-\operatorname{sign}(\langle a_i,u\rangle)$. We have $m=20\,000$, $n=1\,000$ and $\tau=1$. We set $K\leftarrow2$ and $\eta\leftarrow10^{-3/2}$ in AdaCSFW, $\eta\leftarrow10^{-1}$ in AdaGrad, and $\eta\leftarrow10^{-2}$ in AMSGrad. The results are presented in Figure~\ref{fig:svm}.

\begin{figure}[h]
\vspace{2mm}
\centering{\includegraphics[scale=0.6]{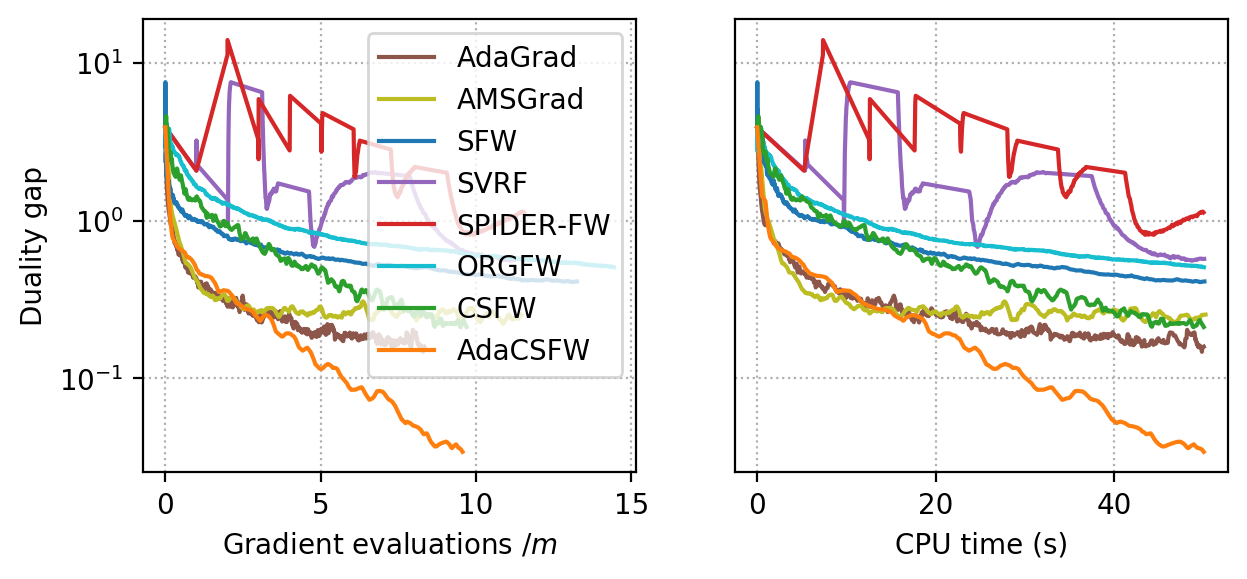}}
\caption{Support vector classification on a synthetic dataset.}
\label{fig:svm}
\end{figure}

\paragraph{Linear regression.}\label{para:lin} We consider a linear regression experiment on the YearPredictionMSD dataset \citep{bertin11}, available at \url{https://archive.ics.uci.edu/ml/datasets/YearPredictionMSD}. The goal is to predict the release years $y_1,\ldots,y_m$ of songs from their audio features $a_1,\ldots,a_m\in\mathbb{R}^n$. We include a sparsity-inducing constraint via the $\ell_1$-norm:
\begin{align*}
\min_{x\in\mathbb{R}^n}\;&\frac{1}{m}\sum_{i=1}^m(y_i-\langle a_i,x\rangle)^2\\
\text{s.t.}\;&\|x\|_1\leq100.
\end{align*}
We have $m=463\,715$ and $n=90$. We set $K\leftarrow2$ and $\eta\leftarrow10^{1/2}$ in AdaSVRF, $\eta\leftarrow10^{-1/2}$ in AdaGrad, and $\eta\leftarrow10^{-3/2}$ in AMSGrad. The results are presented in Figure~\ref{fig:lin}.

\begin{figure}[H]
\vspace{2mm}
\centering{\includegraphics[scale=0.6]{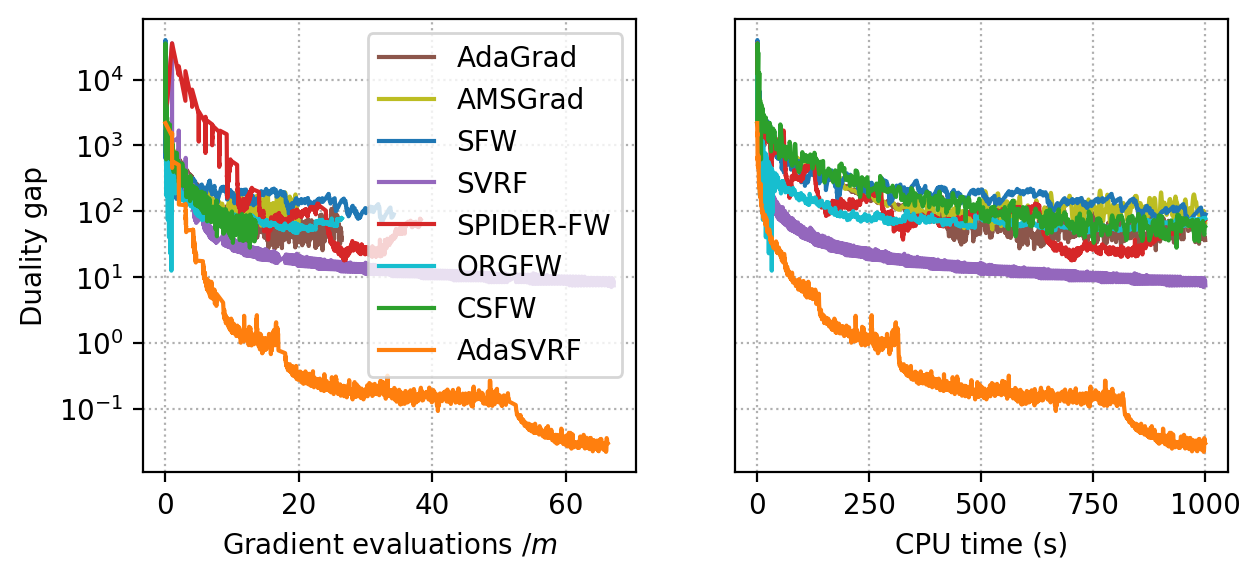}}
\caption{Linear regression on the YearPredictionMSD dataset.}
\label{fig:lin}
\end{figure}

\paragraph{Logistic regression.}\label{para:log} We consider a text categorization experiment on the RCV1 dataset \citep{lewis04}. We use the preprocessed version for binary classification from the LIBSVM library \citep{libsvm}, available at \url{https://www.csie.ntu.edu.tw/~cjlin/libsvmtools/datasets/binary.html#rcv1.binary}, and adopt a logistic regression model with a sparsity-inducing constraint via the $\ell_1$-norm:
\begin{align*}
\min_{x\in\mathbb{R}^n}\;&\frac{1}{m}\sum_{i=1}^m\ln(1+\exp(-y_i\langle a_i,x\rangle))\\
\text{s.t.}\;&\|x\|_1\leq100,
\end{align*}
where $y_i\in\{-1,+1\}$. We have $m=20\,242$ and $n=47\,236$. We set $K\leftarrow5$ and $\eta\leftarrow10^2$ in AdaCSFW, $\eta\leftarrow1$ in AdaGrad, and $\eta\leftarrow10^{-1}$ in AMSGrad. The results are presented in Figure~\ref{fig:log}.

\begin{figure}[h]
\vspace{2mm}
\centering{\includegraphics[scale=0.6]{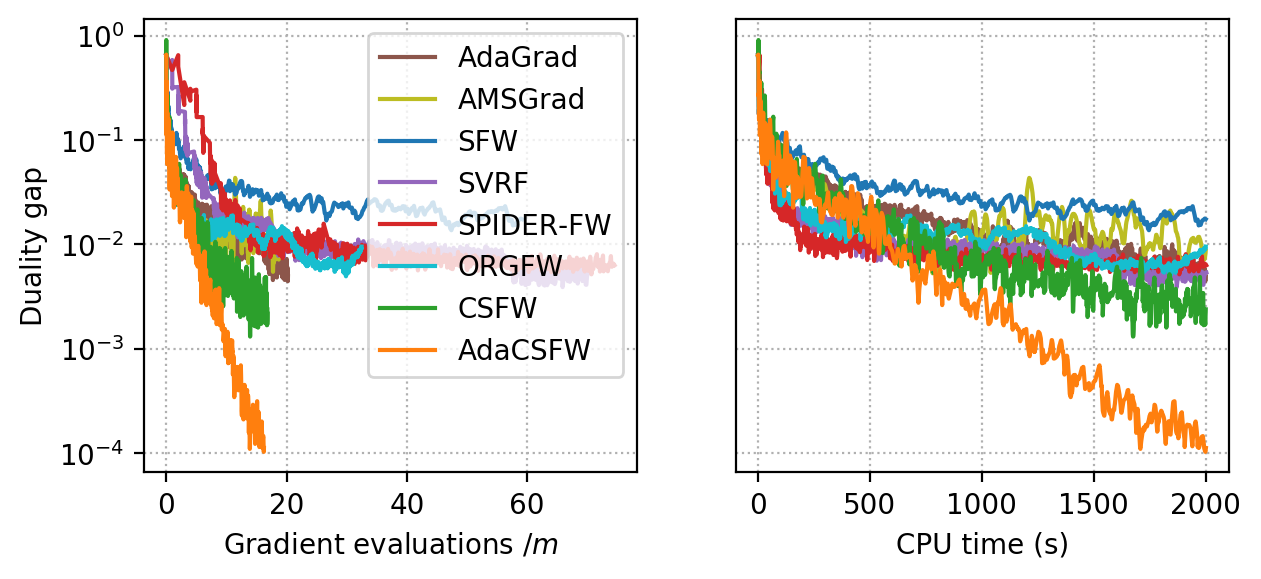}}
\caption{Logistic regression on the RCV1 dataset.}
\label{fig:log}
\end{figure}

\subsection{Nonconvex objectives}
\label{sec:ncvx}

We compare the algorithms on the training of two neural networks. CSFW is not applicable here. Analyses of these algorithms in the nonconvex setting are provided in \citet{reddi16,yurtsever19,orgfw20}. Since variance reduction can be ineffective in the training of deep neural networks \citep{defazio19}, we run Template~\ref{adafw} as AdaSFW only. However, for completeness and transparency, we still compare to the variance-reduced methods, for which we apply \emph{transform locking} as recommended in \citet{defazio19}. We propose a variant of AdaSFW with momentum inspired by Adam and AMSGrad, named AdamSFW; see Appendix~\ref{apx:adamsfw}. In line with the practice of deep learning, we use constant batch-sizes in all algorithms and every hyperparameter is tuned using the same methodology. Experiments were logged with Weights \& Biases \citep{wandb20}. The results are averaged over $5$ runs and the shaded areas represent $\pm1$ standard deviation.

\paragraph{IMDB dataset.}\label{para:imdb} We train a neural network for sentiment analysis on the IMDB dataset \citep{maas11}
for $20$ epochs. We use the 8\,185 subword representation from TensorFlow, available at \url{https://www.tensorflow.org/datasets/catalog/imdb_reviews#imdb_reviewssubwords8k}. The neural network has one fully-connected hidden layer of $64$ units and ReLU activations. Each layer is constrained into an $\ell_\infty$-ball with $\ell_2$-diameter equal to $6$ times the expected $\ell_2$-norm of the Glorot uniform initialized values. We set $K\leftarrow2$ and $\eta\leftarrow10^{-5/2}$ in AdaSFW, $K\leftarrow5$ and $\eta\leftarrow10^{-3}$ in AdamSFW, $\eta\leftarrow10^{-2}$ in AdaGrad, and $\eta\leftarrow10^{-7/2}$ in AMSGrad.
The results are presented in Figure~\ref{fig:imdb}.

\begin{figure}[h]
\vspace{2mm}
\centering{\includegraphics[scale=0.6]{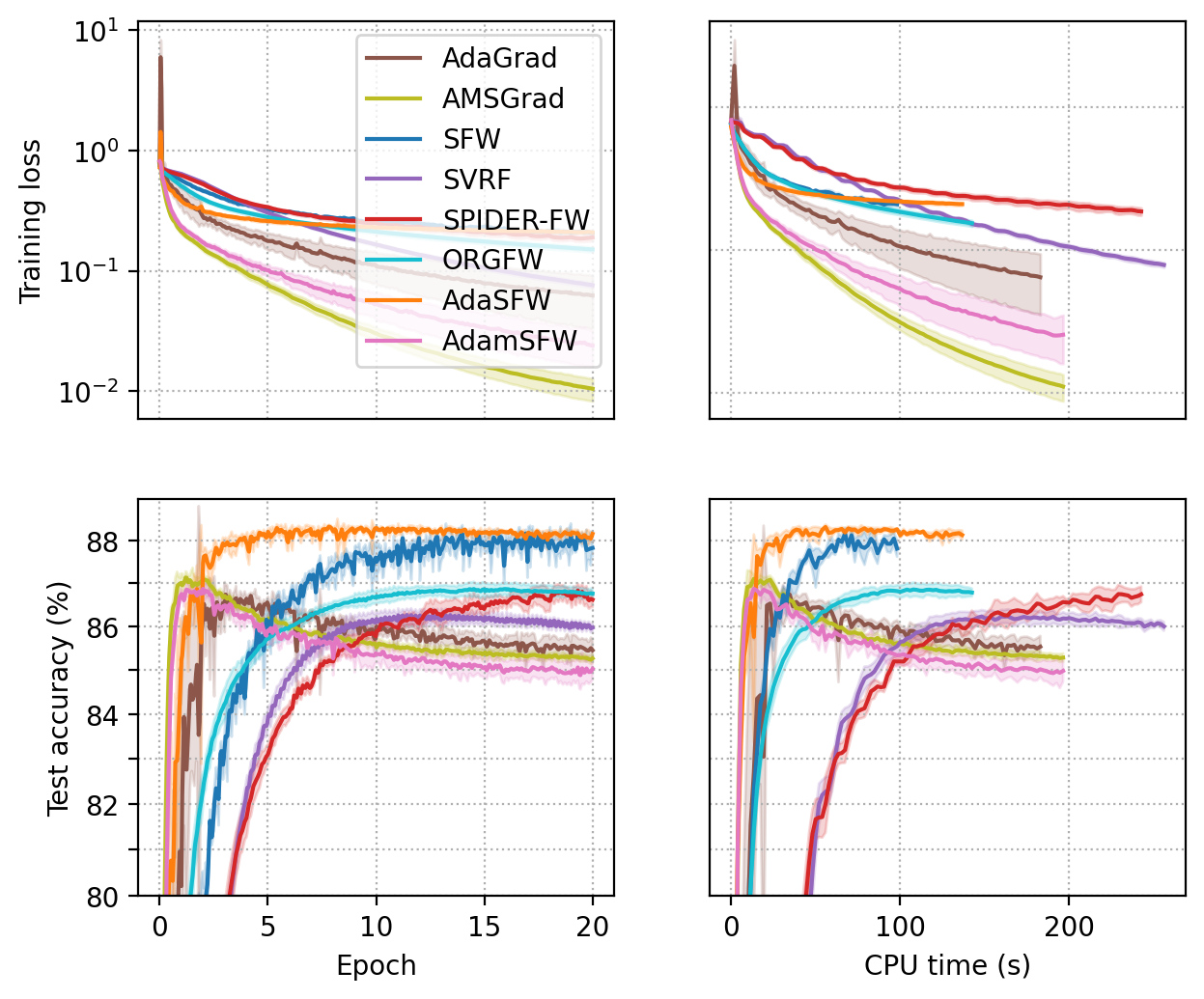}}
\caption{Neural network with one fully-connected hidden layer on the IMDB dataset.}
\label{fig:imdb}
\end{figure}

AdaSFW provides the best test performance, converging both very fast and to the highest accuracy on the test set, despite optimizing slowly on the training set. The vanilla SFW provides a better test accuracy than its variants SVRF, SPIDER-FW, and ORGFW. AdaGrad, AMSGrad, and AdamSFW optimize very fast on the training set and reach their highest test accuracy early on, which can be favorable if we consider using early stopping.

\paragraph{CIFAR-10 dataset.}\label{para:cifar} We train a convolutional neural network (CNN) for image classification on the CIFAR-10 dataset \citep{cifar10}, available at \url{https://www.cs.toronto.edu/~kriz/cifar.html}, for $100$ epochs. It has three $3\times3$ convolutional layers with $32$, $64$, and $64$ channels respectively, two $2\times2$ max-pooling layers, one fully-connected hidden layer of $64$ units, and ReLU activations. Each layer is constrained into an $\ell_\infty$-ball with $\ell_2$-diameter equal to $200$ times the expected $\ell_2$-norm of the Glorot uniform initialized values. We set $K\leftarrow10$ and $\eta\leftarrow10^{-3/2}$ in AdaSFW, $K\leftarrow5$ and $\eta\leftarrow10^{-7/2}$ in AdamSFW, $\eta\leftarrow10^{-2}$ in AdaGrad, and $\eta\leftarrow10^{-7/2}$ in AMSGrad. The results are presented in Figure~\ref{fig:cifar}.

\begin{figure}[H]
\vspace{2mm}
\centering{\includegraphics[scale=0.6]{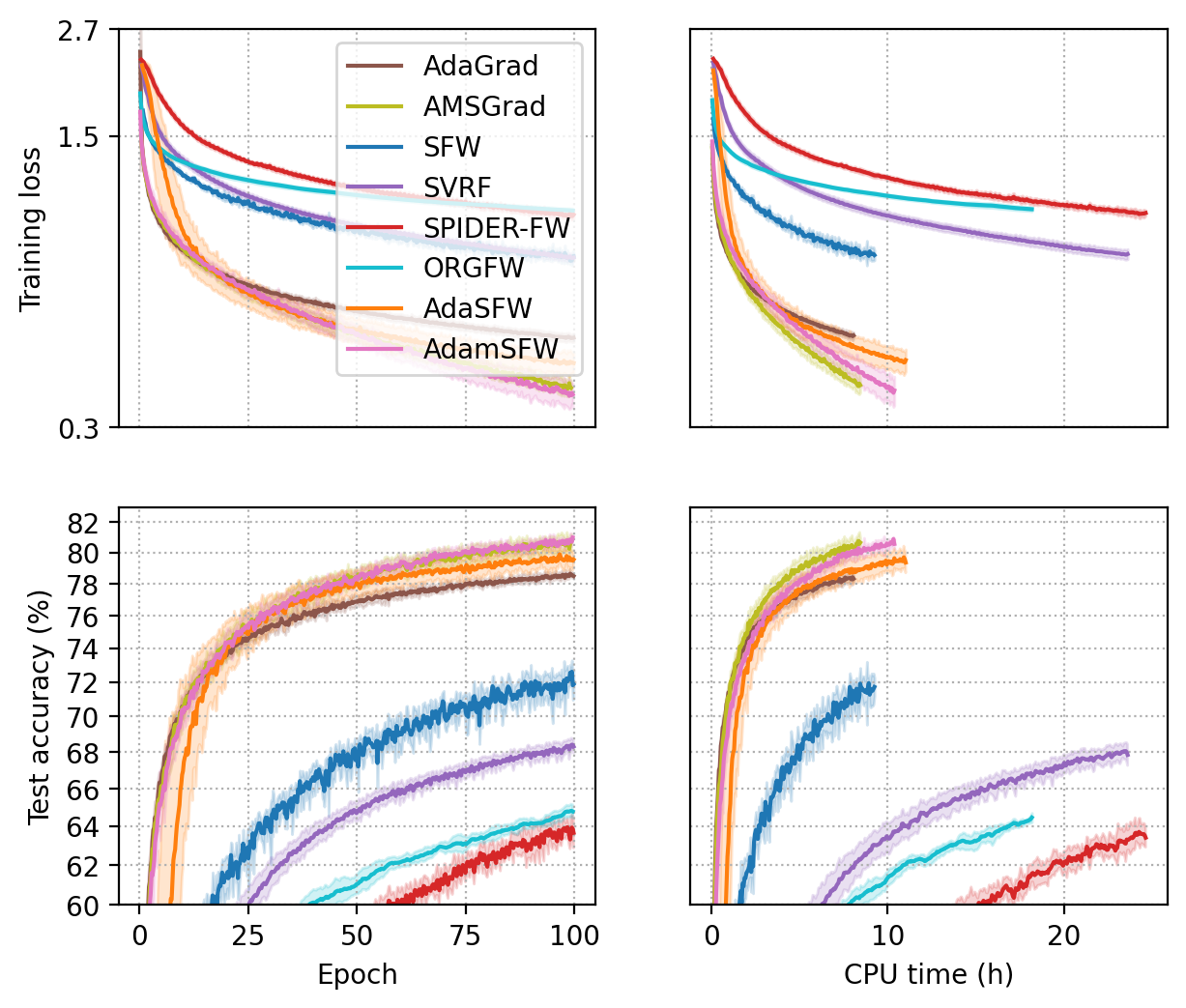}}
\caption{Convolutional neural network on the CIFAR-10 dataset.}
\label{fig:cifar}
\end{figure}

Here, AdaSFW and AdaGrad, and AdamSFW and AMSGrad, have similar performances respectively. Among projection-free algorithms though, AdaSFW and AdamSFW strongly outperform the other algorithms, while, once again, SVRF, SPIDER-FW, and ORGFW perform worse than the vanilla SFW.

\section{Final remarks}
\label{sec:final}

We have proposed a method for large-scale constrained optimization that augments stochastic Frank-Wolfe algorithms through adaptive gradients. We provided theoretical guarantees and demonstrated its computational advantage over the state-of-the-art stochastic Frank-Wolfe algorithms in a wide range of experiments with both convex and nonconvex objectives. On the training of neural networks, our method is the only projection-free algorithm to improve the performance of the vanilla SFW.

We also demonstrated the computational advantage of our method over adaptive gradient algorithms for constrained optimization. This may be an interesting area for future research as adaptive gradient algorithms have proven successful on a variety of tasks, usually addressed as unconstrained problems. Furthermore, the computational advantage may be even more pronounced on instances involving contraint sets that are more complex than the $\ell_1$-ball and the $\ell_\infty$-ball, as the non-Euclidean projection required at each iteration in adaptive gradient algorithms can be prohibitively expensive, while our method is projection-free and still leverages adaptive gradients.

\subsection*{Acknowledgments}

Research reported in this paper was partially supported by Deutsche Forschungsgemeinschaft (DFG) through the DFG Cluster of Excellence MATH+ and the Research Campus MODAL funded by the German Federal Ministry of Education and Research (fund numbers 05M14ZAM, 05M20ZBM).

\bibliographystyle{abbrvnat}
{\small\bibliography{biblio}}

\begin{thebibliography}{39}
\providecommand{\natexlab}[1]{#1}
\providecommand{\url}[1]{\texttt{#1}}
\expandafter\ifx\csname urlstyle\endcsname\relax
  \providecommand{\doi}[1]{doi: #1}\else
  \providecommand{\doi}{doi: \begingroup \urlstyle{rm}\Url}\fi

\bibitem[Argyriou et~al.(2014)Argyriou, Signoretto, and Suykens]{argyriou14}
A.~Argyriou, M.~Signoretto, and J.~A.~K. Suykens.
\newblock Hybrid conditional gradient-smoothing algorithms with applications to
  sparse and low rank regularization.
\newblock In \emph{Regularization, Optimization, Kernels, and Support Vector
  Machines}, pages 53--82. Chapman {\&} Hall/CRC, 2014.

\bibitem[Bertin-Mahieux et~al.(2011)Bertin-Mahieux, Ellis, Whitman, and
  Lamere]{bertin11}
T.~Bertin-Mahieux, D.~P.~W. Ellis, B.~Whitman, and P.~Lamere.
\newblock The {M}illion {S}ong dataset.
\newblock In \emph{Proceedings of the 12th International Conference on Music
  Information Retrieval}, 2011.

\bibitem[Biewald(2020)]{wandb20}
L.~Biewald.
\newblock Experiment tracking with weights and biases, 2020.
\newblock Software available at \url{https://www.wandb.com}.

\bibitem[Chang and Lin(2011)]{libsvm}
C.-C. Chang and C.-J. Lin.
\newblock {LIBSVM}: A library for support vector machines.
\newblock \emph{ACM Transactions on Intelligent Systems and Technology},
  2\penalty0 (3):\penalty0 1--27, 2011.

\bibitem[Combettes and Pokutta(2020)]{combettes20}
C.~W. Combettes and S.~Pokutta.
\newblock Boosting {F}rank-{W}olfe by chasing gradients.
\newblock In \emph{Proceedings of the 37th International Conference on Machine
  Learning}, pages 2111--2121, 2020.

\bibitem[Combettes and Pokutta(2021)]{combettes21}
C.~W. Combettes and S.~Pokutta.
\newblock Complexity of linear minimization and projection on some sets.
\newblock \emph{arXiv preprint arXiv:2101.10040}, 2021.

\bibitem[Cutkosky and Orabona(2019)]{cutkosky19}
A.~Cutkosky and F.~Orabona.
\newblock Momentum-based variance reduction in non-convex {SGD}.
\newblock In \emph{Advances in Neural Information Processing Systems},
  volume~32, pages 15236--15245, 2019.

\bibitem[Dean et~al.(2012)Dean, Corrado, Monga, Chen, Devin, Mao, Ranzato,
  Senior, Tucker, Yang, Le, and Ng]{dean12}
J.~Dean, G.~Corrado, R.~Monga, K.~Chen, M.~Devin, M.~Mao, M.~Ranzato,
  A.~Senior, P.~Tucker, K.~Yang, Q.~V. Le, and A.~Y. Ng.
\newblock Large scale distributed deep networks.
\newblock In \emph{Advances in Neural Information Processing Systems},
  volume~25, pages 1223--1231, 2012.

\bibitem[Defazio and Bottou(2019)]{defazio19}
A.~Defazio and L.~Bottou.
\newblock On the ineffectiveness of variance reduced optimization for deep
  learning.
\newblock In \emph{Advances in Neural Information Processing Systems},
  volume~32, pages 1755--1765, 2019.

\bibitem[Duchi(2018)]{duchi18}
J.~C. Duchi.
\newblock Introductory lectures on stochastic optimization.
\newblock In \emph{The Mathematics of Data}. American Mathematical Society,
  2018.

\bibitem[Duchi et~al.(2011)Duchi, Hazan, and Singer]{adagrad11}
J.~C. Duchi, E.~Hazan, and Y.~Singer.
\newblock Adaptive subgradient methods for online learning and stochastic
  optimization.
\newblock \emph{Journal of Machine Learning Research}, 12\penalty0
  (61):\penalty0 2121--2159, 2011.

\bibitem[Fang et~al.(2018)Fang, Li, Lin, and Zhang]{fang18}
C.~Fang, C.~J. Li, Z.~Lin, and T.~Zhang.
\newblock {SPIDER}: Near-optimal non-convex optimization via stochastic
  path-integrated differential estimator.
\newblock In \emph{Advances in Neural Information Processing Systems},
  volume~31, pages 689--699, 2018.

\bibitem[Frank and Wolfe(1956)]{fw56}
M.~Frank and P.~Wolfe.
\newblock An algorithm for quadratic programming.
\newblock \emph{Naval Research Logistics Quarterly}, 3\penalty0 (1-2):\penalty0
  95--110, 1956.

\bibitem[Hazan and Luo(2016)]{hazan16}
E.~Hazan and H.~Luo.
\newblock Variance-reduced and projection-free stochastic optimization.
\newblock In \emph{Proceedings of the 33rd International Conference on Machine
  Learning}, pages 1263--1271, 2016.

\bibitem[Jaggi(2013)]{jaggi13fw}
M.~Jaggi.
\newblock Revisiting {F}rank-{W}olfe: Projection-free sparse convex
  optimization.
\newblock In \emph{Proceedings of the 30th International Conference on Machine
  Learning}, pages 427--435, 2013.

\bibitem[Johnson and Zhang(2013)]{svrg13}
R.~Johnson and T.~Zhang.
\newblock Accelerating stochastic gradient descent using predictive variance
  reduction.
\newblock In \emph{Advances in Neural Information Processing Systems},
  volume~26, pages 315--323, 2013.

\bibitem[Keskar and Socher(2017)]{keskar17}
N.~S. Keskar and R.~Socher.
\newblock Improving generalization performance by switching from {A}dam to
  {SGD}.
\newblock \emph{arXiv preprint arXiv:1712.07628}, 2017.

\bibitem[Kingma and Ba(2015)]{adam15}
D.~P. Kingma and J.~Ba.
\newblock Adam: A method for stochastic optimization.
\newblock In \emph{Proceedings of the 3rd International Conference on Learning
  Representations}, 2015.

\bibitem[Krizhevsky(2009)]{cifar10}
A.~Krizhevsky.
\newblock Learning multiple layers of features from tiny images.
\newblock Master's thesis, 2009.

\bibitem[Lacoste-Julien(2016)]{lj16}
S.~Lacoste-Julien.
\newblock Convergence rate of {F}rank-{W}olfe for non-convex objectives.
\newblock \emph{arXiv preprint arXiv:1607.00345}, 2016.

\bibitem[Lan and Zhou(2016)]{lan16cgs}
G.~Lan and Y.~Zhou.
\newblock Conditional gradient sliding for convex optimization.
\newblock \emph{SIAM Journal on Optimization}, 26\penalty0 (2):\penalty0
  1379--1409, 2016.

\bibitem[Levitin and Polyak(1966)]{polyak66cg}
E.~S. Levitin and B.~T. Polyak.
\newblock Constrained minimization methods.
\newblock \emph{USSR Computational Mathematics and Mathematical Physics},
  6\penalty0 (5):\penalty0 1--50, 1966.

\bibitem[Lewis et~al.(2004)Lewis, Yang, Rose, and Li]{lewis04}
D.~D. Lewis, Y.~Yang, T.~G. Rose, and F.~Li.
\newblock {RCV1}: A new benchmark collection for text categorization research.
\newblock \emph{Journal of Machine Learning Research}, 5:\penalty0 361--397,
  2004.

\bibitem[Luo et~al.(2019)Luo, Xiong, Liu, and Sun]{adabound19}
L.~Luo, Y.~Xiong, Y.~Liu, and X.~Sun.
\newblock Adaptive gradient methods with dynamic bound of learning rate.
\newblock In \emph{Proceedings of the 7th International Conference on Learning
  Representations}, 2019.

\bibitem[Maas et~al.(2011)Maas, Daly, Pham, Huang, Ng, and Potts]{maas11}
A.~L. Maas, R.~E. Daly, P.~T. Pham, D.~Huang, A.~Y. Ng, and C.~Potts.
\newblock Learning word vectors for sentiment analysis.
\newblock In \emph{Proceedings of the 49th Annual Meeting of the Association
  for Computational Linguistics: Human Language Technologies - Volume 1}, pages
  142--150, 2011.

\bibitem[McMahan and Streeter(2010)]{mcmahan10}
H.~B. McMahan and M.~Streeter.
\newblock Adaptive bound optimization for online convex optimization.
\newblock In \emph{Proceedings of the 23rd Annual Conference on Learning
  Theory}, 2010.

\bibitem[N{\'e}giar et~al.(2020)N{\'e}giar, Dresdner, Tsai, El~Ghaoui,
  Locatello, Freund, and Pedregosa]{negiar20}
G.~N{\'e}giar, G.~Dresdner, A.~Y.-T. Tsai, L.~El~Ghaoui, F.~Locatello, R.~M.
  Freund, and F.~Pedregosa.
\newblock Stochastic {F}rank-{W}olfe for constrained finite-sum minimization.
\newblock In \emph{Proceedings of the 37th International Conference on Machine
  Learning}, pages 7253--7262, 2020.

\bibitem[Polyak(1987)]{polyak87}
B.~T. Polyak.
\newblock \emph{Introduction to Optimization}.
\newblock Optimization Software, 1987.

\bibitem[Reddi et~al.(2016)Reddi, Sra, P{\'o}czos, and Smola]{reddi16}
S.~J. Reddi, S.~Sra, B.~P{\'o}czos, and A.~Smola.
\newblock Stochastic {F}rank-{W}olfe methods for nonconvex optimization.
\newblock In \emph{54th Annual Allerton Conference on Communication, Control,
  and Computing}, pages 1244--1251, 2016.

\bibitem[Reddi et~al.(2018)Reddi, Kale, and Kumar]{reddi18}
S.~J. Reddi, S.~Kale, and S.~Kumar.
\newblock On the convergence of {A}dam and beyond.
\newblock In \emph{Proceedings of the 6th International Conference on Learning
  Representations}, 2018.

\bibitem[Schmidt et~al.(2017)Schmidt, Le~Roux, and Bach]{schmidt17}
M.~Schmidt, N.~Le~Roux, and F.~Bach.
\newblock Minimizing finite sums with the stochastic average gradient.
\newblock \emph{Mathematical Programming}, 162\penalty0 (1-2):\penalty0
  83--112, 2017.

\bibitem[Shen et~al.(2019)Shen, Fang, Zhao, Huang, and Qian]{shen19}
Z.~Shen, C.~Fang, P.~Zhao, J.~Huang, and H.~Qian.
\newblock Complexities in projection-free stochastic non-convex minimization.
\newblock In \emph{Proceedings of the 22nd International Conference on
  Artificial Intelligence and Statistics}, pages 2868--2876, 2019.

\bibitem[Tieleman and Hinton(2012)]{rmsprop12}
T.~Tieleman and G.~Hinton.
\newblock Lecture 6e -- rmsprop: divide the gradient by a running average of
  its recent magnitude.
\newblock \emph{COURSERA: Neural Networks for Machine Learning}, 4\penalty0
  (2):\penalty0 26--31, 2012.

\bibitem[Wilson et~al.(2017)Wilson, Roelofs, Stern, Srebro, and
  Recht]{wilson17}
A.~C. Wilson, R.~Roelofs, M.~Stern, N.~Srebro, and B.~Recht.
\newblock The marginal value of adaptive gradient methods in machine learning.
\newblock In \emph{Advances in Neural Information Processing Systems},
  volume~30, pages 4148--4158, 2017.

\bibitem[Xie et~al.(2020)Xie, Shen, Zhang, Qian, and Wang]{orgfw20}
J.~Xie, Z.~Shen, C.~Zhang, H.~Qian, and B.~Wang.
\newblock Efficient projection-free online methods with stochastic recursive
  gradient.
\newblock In \emph{Proceedings of the 34th AAAI Conference on Artificial
  Intelligence}, pages 6446--6453, 2020.

\bibitem[Yurtsever et~al.(2019)Yurtsever, Sra, and Cevher]{yurtsever19}
A.~Yurtsever, S.~Sra, and V.~Cevher.
\newblock Conditional gradient methods via stochastic path-integrated
  differential estimator.
\newblock In \emph{Proceedings of the 36th International Conference on Machine
  Learning}, pages 7282--7291, 2019.

\bibitem[Zeiler(2012)]{adadelta12}
M.~D. Zeiler.
\newblock Ada{D}elta: An adaptive learning rate method.
\newblock \emph{arXiv preprint arXiv:1212.5701}, 2012.

\bibitem[Zhang et~al.(2013)Zhang, Mahdavi, and Jin]{zhang13}
L.~Zhang, M.~Mahdavi, and R.~Jin.
\newblock Linear convergence with condition number independent access of full
  gradients.
\newblock In \emph{Advances in Neural Information Processing Systems},
  volume~26, pages 980--988, 2013.

\bibitem[Zhang and Oles(2001)]{zhang01}
T.~Zhang and F.~J. Oles.
\newblock Text categorization based on regularized linear classification
  methods.
\newblock \emph{Information Retrieval}, 4:\penalty0 5--31, 2001.

\end{thebibliography}

\clearpage
\appendix

\section{AdamSFW: AdaSFW with momentum}
\label{apx:adamsfw}

In Algorithm~\ref{adamsfw}, inspired by Adam \citep{adam15} and AMSGrad \citep{reddi18}, we propose a variant of AdaSFW (Algorithm~\ref{adasfw}) with momentum which we used in our neural network training experiments (Section~\ref{sec:ncvx}). The batch-size $b\in\mathbb{N}\backslash\{0\}$ and the learning rate $\eta>0$ could be chosen as time-varying. Following \citet{reddi18}, we require $\beta_m<\sqrt{\beta_s}$, with default values $\beta_m\leftarrow0.9$ and $\beta_s\leftarrow0.99$ or $\beta_s\leftarrow0.999$. All operations in Line~\ref{ss} are entry-wise in $\mathbb{R}^n$. Notice the presence of the momentum term $m_t$ instead of $\tilde{\nabla}f(x_t)$ in Line~\ref{adam:sub}: the subproblem addressed by AdamSFW is
\begin{align*}
 \min_{x\in\mathcal{C}}\left\{Q_t(x)\coloneqq f(x_t)+\langle m_t,x-x_t\rangle+\frac{1}{2\eta}\|x-x_t\|_{H_t}^2\right\}.
\end{align*}

\vspace{1mm}
\begin{algorithm}[h]
\caption{AdamSFW}
\label{adamsfw}
\begin{algorithmic}[1]
\REQUIRE Start point $x_0\in\mathcal{C}$, batch-size $b\in\mathbb{N}\backslash\{0\}$, momentum parameters $\beta_m,\beta_s\in\left]0,1\right[$, offset $\delta>0$, number of inner iterations $K\in\mathbb{N}\backslash\{0\}$, learning rate $\eta>0$.
\STATE$m_{-1},s_{-1},\bar{s}_{-1}\leftarrow0,0,0$
\FOR{$t=0$ \textbf{to} $T-1$}
\STATE$i_1,\ldots,i_b\overset{\text{i.i.d.}}{\sim}\mathcal{U}(\llbracket1,m\rrbracket)$
\STATE$\tilde{\nabla}f(x_t)\leftarrow\displaystyle\frac{1}{b}\sum_{i=i_1}^{i_b}\nabla f_i(x_t)$
\STATE$m_t\leftarrow\beta_mm_{t-1}+(1-\beta_m)\tilde{\nabla}f(x_t)$
\STATE$s_t\leftarrow\beta_ss_{t-1}+(1-\beta_s)\tilde{\nabla}f(x_t)^2$
\STATE$\bar{s}_t\leftarrow\max\{\bar{s}_{t-1},s_t\}$
\STATE$H_t\leftarrow\operatorname{diag}(\delta1+\sqrt{\bar{s}_t})$\label{ss}
\STATE$y_0^{(t)}\leftarrow x_t$
\FOR{$k=0$ \textbf{to} $K-1$}
\STATE$\nabla Q_t(y_k^{(t)})\leftarrow m_t+\displaystyle\frac{1}{\eta}H_t(y_k^{(t)}-x_t)$\label{adam:sub}
\STATE$v_k^{(t)}\leftarrow\argmin\limits_{v\in\mathcal{C}}\langle\nabla Q_t(y_k^{(t)}),v\rangle$
\STATE$\gamma_k^{(t)}\leftarrow\min\left\{\eta\displaystyle\frac{\langle\nabla Q_t(y_k^{(t)}),y_k^{(t)}-v_k^{(t)}\rangle}{\|y_k^{(t)}-v_k^{(t)}\|_{H_t}^2},1\right\}$
\STATE$y_{k+1}\leftarrow y_k^{(t)}+\gamma_k^{(t)}(v_k^{(t)}-y_k^{(t)})$
\ENDFOR
\STATE$x_{t+1}\leftarrow y_K^{(t)}$
\ENDFOR
\end{algorithmic}
\end{algorithm}
\vspace{1mm}

\newpage
\section{Sensitivity to the number $K$ of inner iterations}
\label{apx:K}

We report in Figures~\ref{fig:svm-K}--\ref{fig:cifar-K} the sensitivity to $K$ in the respective computational experiments (Section~\ref{sec:exp}). In most cases, we can see that for large values of $K$, the method becomes less efficient in CPU time. This validates our approach, detailed in Section~\ref{sec:approach}, since $K\gg1$ represents solving the subproblems (almost) completely.

\begin{figure}[h]
\vspace{2mm}
\centerline{\includegraphics[scale=0.58]{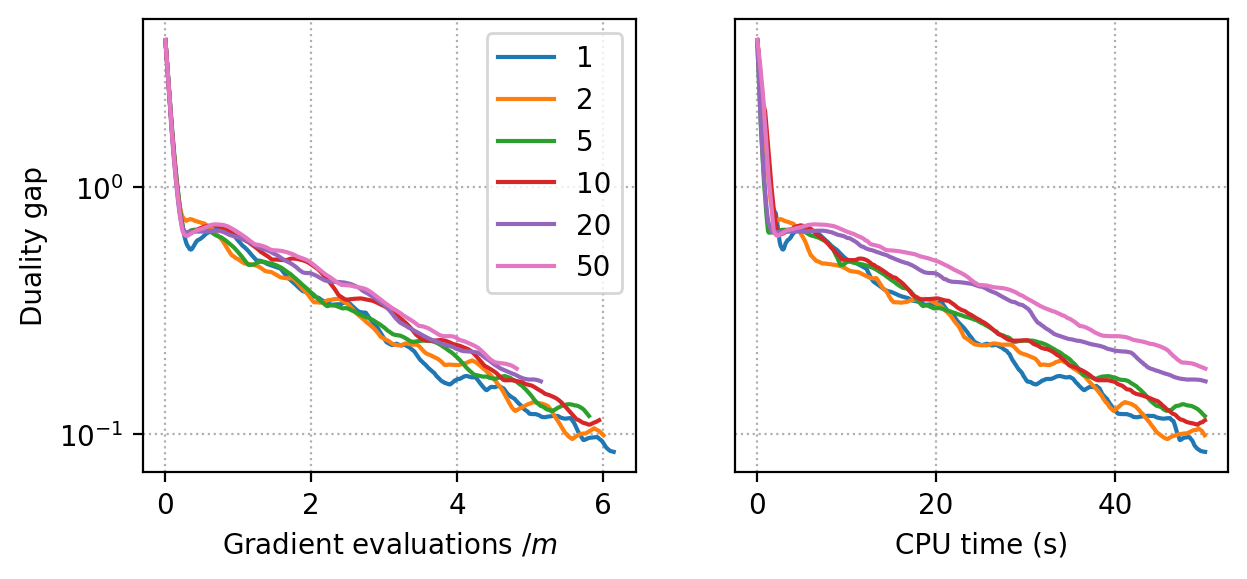}}
\caption{Sensitivity of AdaCSFW to $K$ on the \hyperref[para:svm]{support vector classification} experiment.}
\label{fig:svm-K}
\end{figure}

\begin{figure}[h]
\vspace{2mm}
\centerline{\includegraphics[scale=0.58]{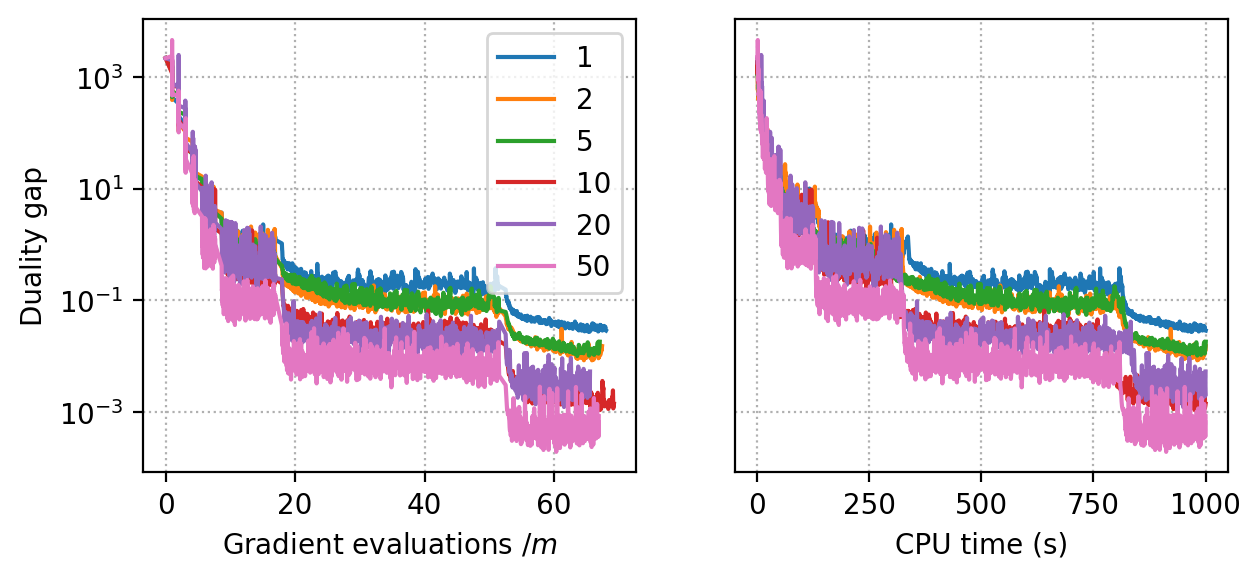}}
\caption{Sensitivity of AdaSVRF to $K$ on the \hyperref[para:lin]{linear regression} experiment.}
\label{fig:lin-K}
\end{figure}

\begin{figure}[h]
\vspace{2mm}
\centerline{\includegraphics[scale=0.58]{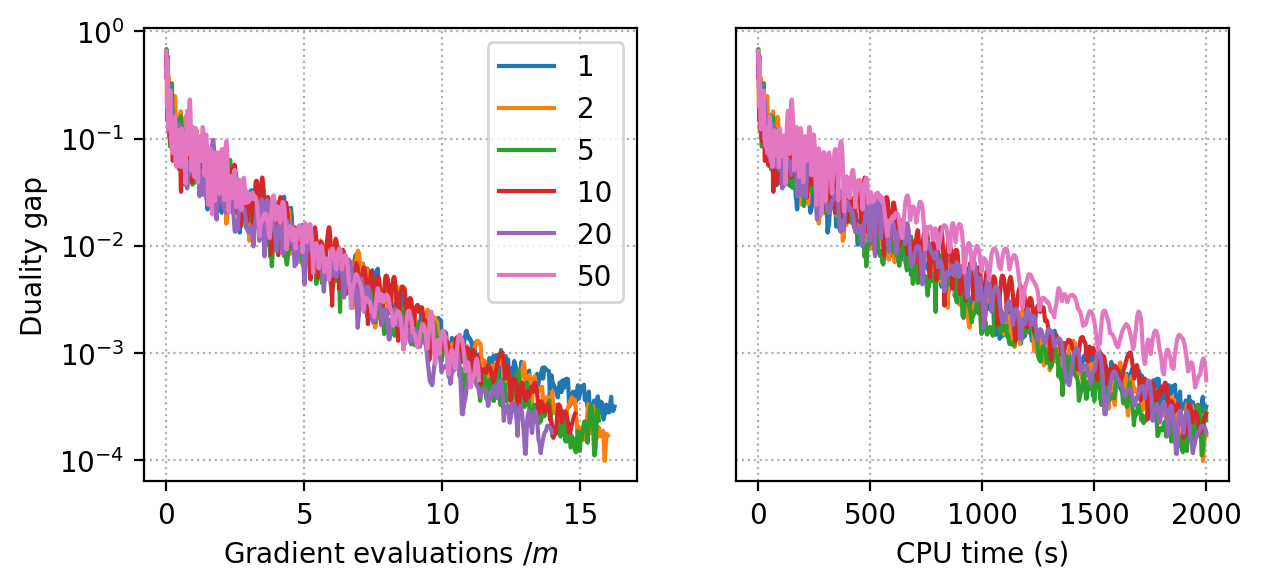}}
\caption{Sensitivity of AdaCSFW to $K$ on the \hyperref[para:log]{logistic regression} experiment.}
\label{fig:log-K}
\end{figure}

\begin{figure}[h]
\vspace{2mm}
\centerline{\includegraphics[scale=0.53]{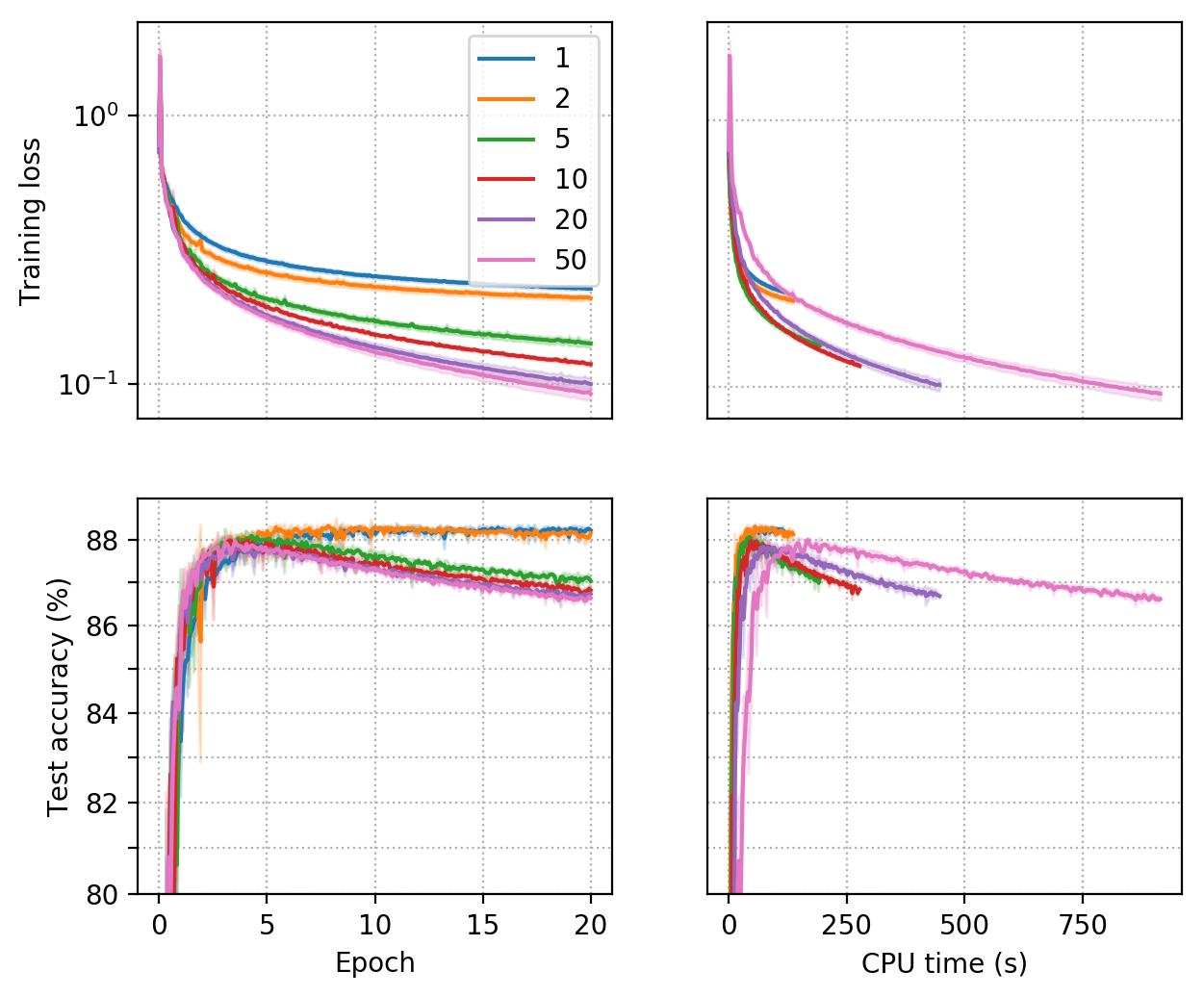}
\hfill\includegraphics[scale=0.53]{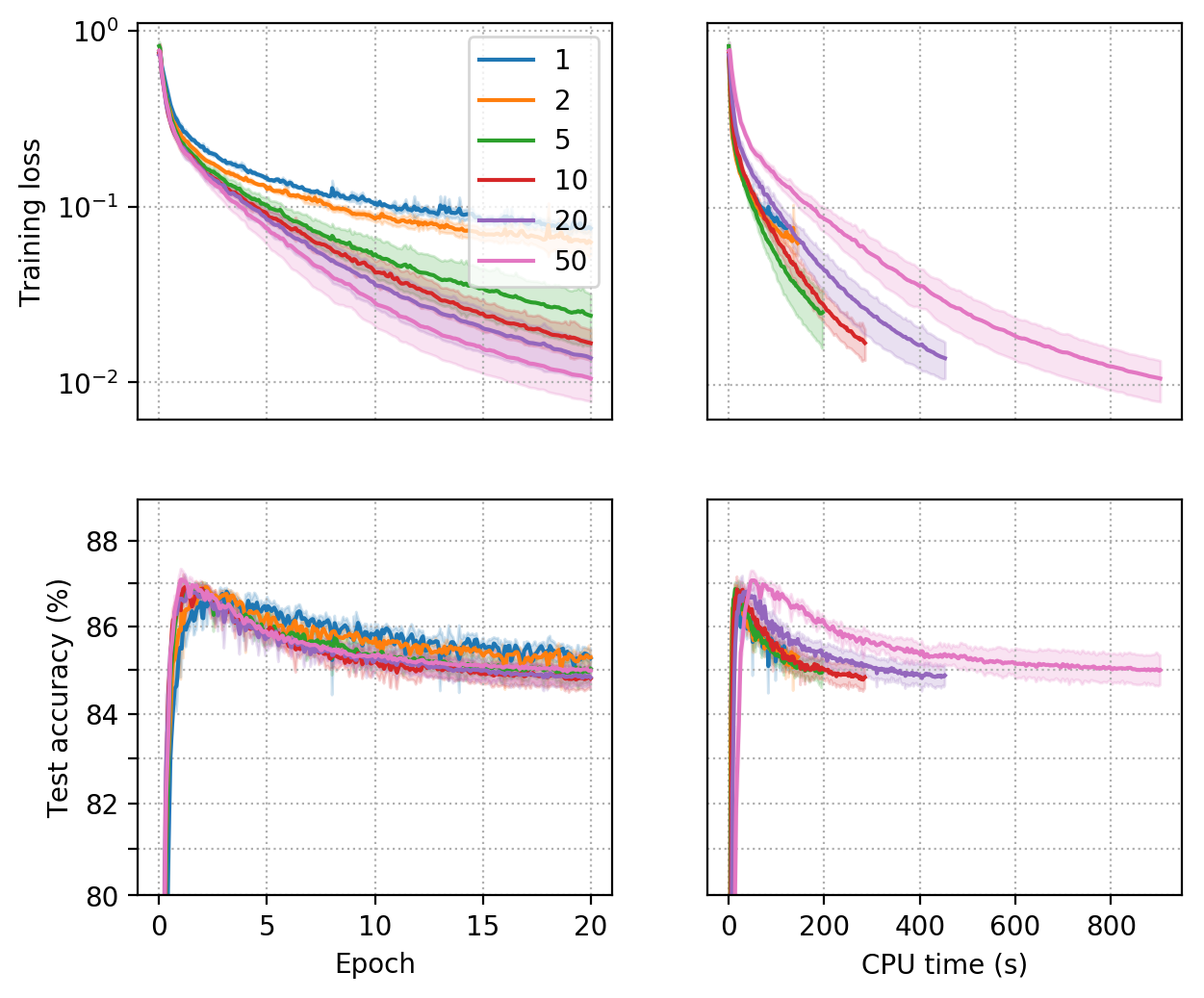}}
\caption{Sensitivity of AdaSFW (left) and AdamSFW (right) to $K$ on the \hyperref[para:imdb]{IMDB dataset} experiment.}
\label{fig:imdb-K}
\end{figure}

\begin{figure}[h]
\vspace{2mm}
\centerline{\includegraphics[scale=0.53]{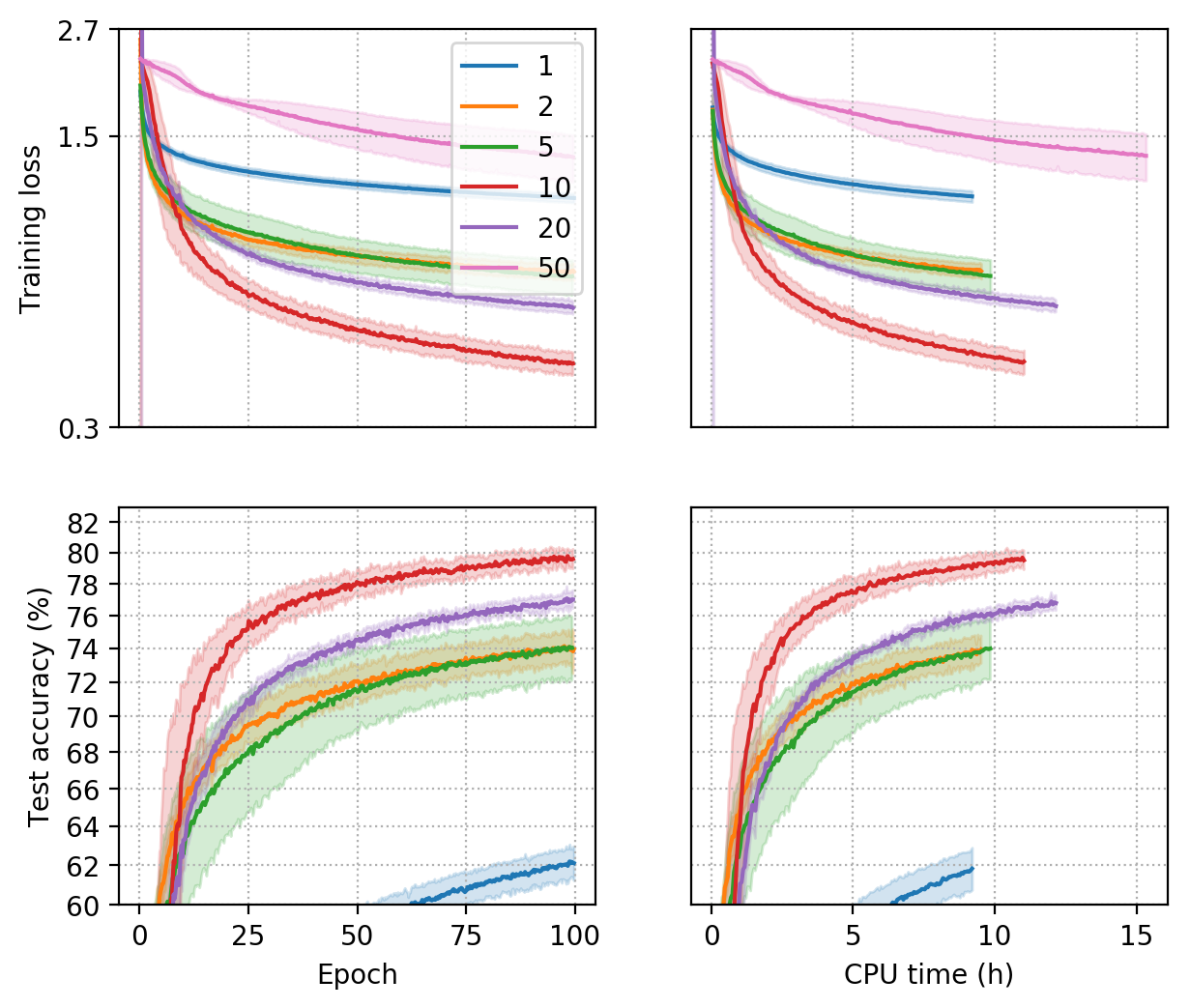}
\hfill\includegraphics[scale=0.53]{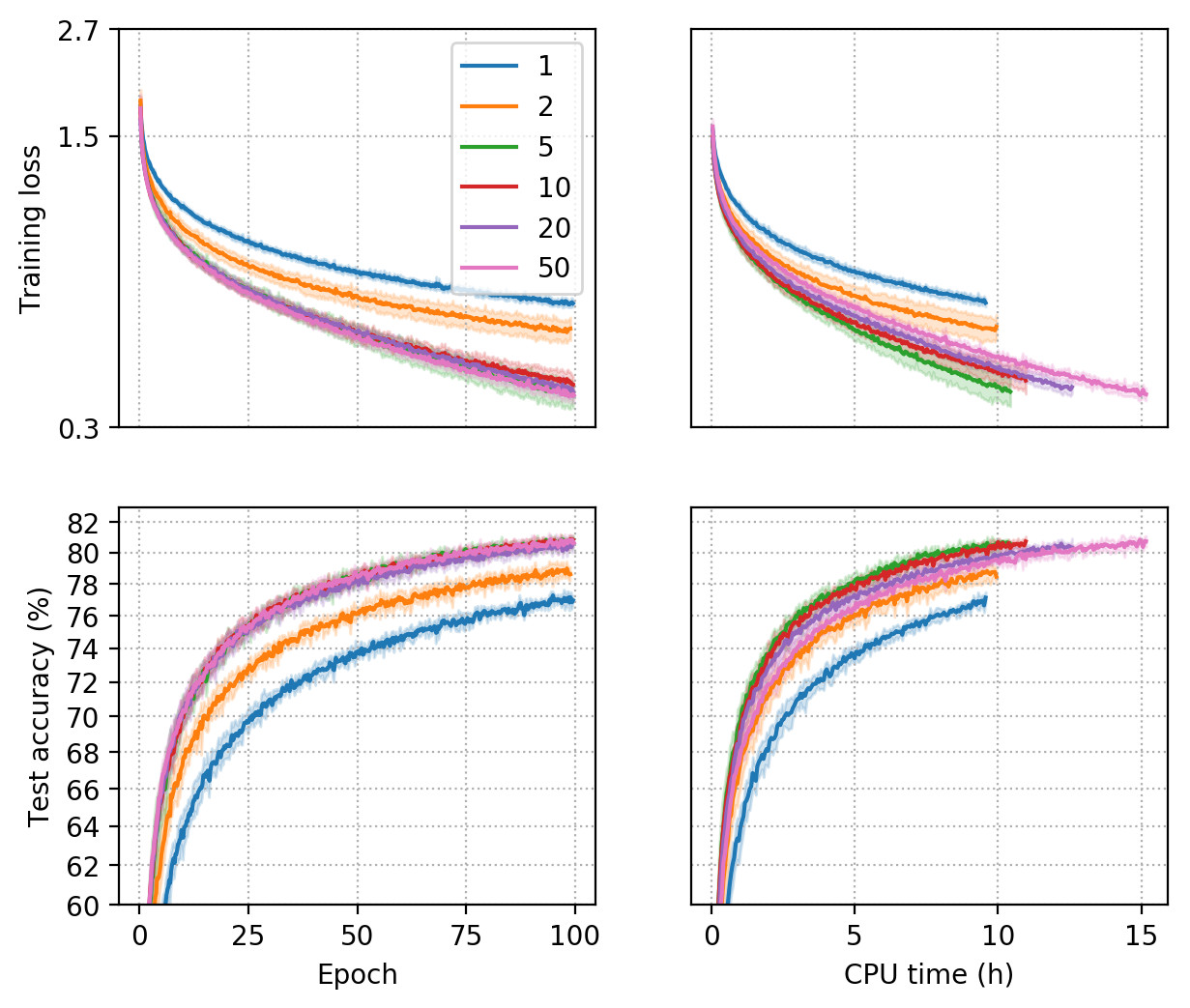}}
\caption{Sensitivity of AdaSFW (left) and AdamSFW (right) to $K$ on the \hyperref[para:cifar]{CIFAR-10 dataset} experiment.}
\label{fig:cifar-K}
\end{figure}

\clearpage
\section{The Frank-Wolfe duality gap}

We report in Fact~\ref{fact:wolfe} some well-known properties of the Frank-Wolfe duality gap \citep{jaggi13fw}.

\begin{fact}
\label{fact:wolfe}
 Let $\mathcal{C}\subset\mathbb{R}^n$ be a compact convex set, $f\colon\mathbb{R}^n\rightarrow\mathbb{R}$ be a smooth function, and $x\in\mathcal{C}$. The Frank-Wolfe duality gap of $f$ at $x$ over $\mathcal{C}$ is $g(x)\coloneqq\max_{v\in\mathcal{C}}\langle\nabla f(x),x-v\rangle$ and satisfies
 \begin{enumerate}[label=(\roman*)]
  \item $g(x)\geq0$,
  \item $g(x)=0\Leftrightarrow x$ is a stationary point,
  \item $f(x)-\min_\mathcal{C}f\leq g(x)$ if $f$ is convex.
 \end{enumerate}
\end{fact}

\begin{proof}
\begin{enumerate}[label=(\roman*)]
\item Let $w\in\argmin_{v\in\mathcal{C}}\langle\nabla f(x),v\rangle$. We have
\begin{align*}
 g(x)
 &=\max_{v\in\mathcal{C}}\langle\nabla f(x),x-v\rangle\\
 &=\langle\nabla f(x),x-w\rangle\\
 &=\langle\nabla f(x),x\rangle-\langle\nabla f(x),w\rangle\\
 &\geq0,
\end{align*}
by definition of $w$.
\item We have
\begin{align*}
 0
 &=g(x)\\
 &=\max_{v\in\mathcal{C}}\langle\nabla f(x),x-v\rangle\\
 &\geq\langle\nabla f(x),x-y\rangle,
\end{align*}
for all $y\in\mathcal{C}$. Therefore, there exists no descent direction for $f$ at $x$ over $\mathcal{C}$. The converse is trivial.
\item Let $x^*\in\argmin_\mathcal{C}f$. By convexity of $f$,
 \begin{align*}
  f(x)-\min_\mathcal{C}f
  &=f(x)-f(x^*)\\
  &\leq\langle\nabla f(x),x-x^*\rangle\\
  &\leq\max_{v\in\mathcal{C}}\langle\nabla f(x),x-v\rangle\\
  &=g(x),
 \end{align*}
 since $x^*\in\mathcal{C}$.
\end{enumerate}
\end{proof}

\section{Proofs}
\label{apx:proofs}

\subsection{The algorithm}

\begin{lemma}
 \label{lem:sub}
 Consider Template~\ref{adafw} and let $t\in\llbracket0,T-1\rrbracket$. For all $k\in\llbracket0,K-1\rrbracket$,
 \begin{align*}
  Q_t(y_{k+1}^{(t)})
  &=\min_{\gamma\in\left[0,\gamma_t\right]}Q_t(y_k^{(t)}+\gamma(v_k^{(t)}-y_k^{(t)})).
 \end{align*}
 In particular, $Q_t(y_{k+1}^{(t)})\leq Q_t(y_k^{(t)})$ and $Q_t(y_1^{(t)})\leq Q_t(y_0^{(t)}+\gamma_t(v_0^{(t)}-y_0^{(t)}))$.
\end{lemma}

\begin{proof}
 Let $k\in\llbracket0,K-1\rrbracket$ and $\varphi_k^{(t)}\colon\gamma\in\mathbb{R}\mapsto Q_t(y_k^{(t)}+\gamma(v_k^{(t)}-y_k^{(t)}))$. Then $\varphi_k^{(t)}$ is a convex quadratic and is minimized at 
 \begin{align*}
  \gamma^*\coloneqq\eta_t\frac{\langle\nabla Q_t(y_k^{(t)}),y_k^{(t)}-v_k^{(t)}\rangle}{\|y_k^{(t)}-v_k^{(t)}\|_{H_t}^2}.
 \end{align*}
 Since $v_k^{(t)}\in\argmin_{v\in\mathcal{C}}\langle\nabla Q_t(y_k^{(t)}),v\rangle$ and $y_k^{(t)}\in\mathcal{C}$, we have $\langle\nabla Q_t(y_k^{(t)}),y_k^{(t)}-v_k^{(t)}\rangle\geq0$ so $\gamma^*\geq0$. Thus, $\varphi_k^{(t)}$ is a decreasing function over $\left[0,\gamma^*\right]$. Since $\gamma_k^{(t)}=\min\{\gamma^*,\gamma_t\}$, we obtain
 \begin{align*}
  \varphi_k^{(t)}(\gamma_k^{(t)})
  &=\min_{\gamma\in\left[0,\gamma_t\right]}\varphi_k^{(t)}(\gamma),
 \end{align*}
 i.e.,
 \begin{align*}
  Q_t(y_{k+1}^{(t)})
  &=\min_{\gamma\in\left[0,\gamma_t\right]}Q_t(y_k^{(t)}+\gamma(v_k^{(t)}-y_k^{(t)})).
 \end{align*}
\end{proof}

\begin{lemma}
 \label{lem:x}
 Consider Template~\ref{adafw}. For all $t\in\llbracket0,T-1\rrbracket$,
\begin{align*}
 \|x_{t+1}-x_t\|_2
 \leq KD\gamma_t.
\end{align*}
\end{lemma}

\begin{proof}
 Let $t\in\llbracket0,T-1\rrbracket$. We have
 \begin{align*}
  x_{t+1}-x_t
  =y_K^{(t)}-y_0^{(t)}
 \end{align*}
 and, by a straightforward induction on $k\in\llbracket0,K\rrbracket$,
 \begin{align*}
  y_K^{(t)}-y_0^{(t)}
  =\sum_{k=0}^{K-1}\left(\prod_{\ell=k+1}^{K-1}(1-\gamma_\ell^{(t)})\right)\gamma_k^{(t)}(v_k^{(t)}-x_t).
 \end{align*}
 Since for all $k\in\llbracket0,K-1\rrbracket$,
 \begin{align*}
  0\leq\gamma_k^{(t)}\leq\gamma_t\leq1,
 \end{align*}
 we obtain
 \begin{align*}
  \|x_{t+1}-x_t\|_2
  &\leq\sum_{k=0}^{K-1}\left(\prod_{\ell=k+1}^{K-1}(1-\gamma_\ell^{(t)})\right)\gamma_k^{(t)}\|v_k^{(t)}-x_t\|_2\\
  &\leq\sum_{k=0}^{K-1}1\cdot\gamma_t\cdot D\\
  &=K\gamma_tD.
 \end{align*}
\end{proof}

\subsection{SFW with adaptive gradients}

Lemma~\ref{lem:sfw} is adapted from \citet[Appendix~B]{hazan16}.

\begin{lemma}
 \label{lem:sfw}
 Consider AdaSFW (Algorithm~\ref{adasfw}). For all $t\in\llbracket0,T-1\rrbracket$,
 \begin{align*}
  \mathsf{E}[\|\tilde{\nabla}f(x_t)-\nabla f(x_t)\|_2]
  \leq\frac{G}{\sqrt{b_t}}.
 \end{align*}
\end{lemma}

\begin{theorem}[(Theorem~\ref{th:adasfw})]
 Consider AdaSFW (Algorithm~\ref{adasfw}) with $b_t\leftarrow\big(G(t+2)/(LD)\big)^2$, $\eta_t\leftarrow\lambda_t^-/L$, and $\gamma_t\leftarrow2/(t+2)$, and let $\kappa\coloneqq\lambda_0^+/\lambda_0^-$. Then for all $t\in\llbracket1,T\rrbracket$,
 \begin{align*}
  \mathsf{E}[f(x_t)]-\min_\mathcal{C}f
  \leq\frac{2LD^2(K+1+\kappa)}{t+1}.
 \end{align*}
\end{theorem}

\begin{proof}
Let $t\in\llbracket0,T-1\rrbracket$. By~\eqref{h2},
\begin{align}
 \frac{L}{2}\|\cdot\|_2^2
 \leq\frac{L}{2\lambda_t^-}\|\cdot\|_{H_t}^2
 =\frac{1}{2\eta_t}\|\cdot\|_{H_t}^2\label{2-h}
\end{align}
and
\begin{align}
 \frac{1}{2\eta_t}\|\cdot\|_{H_t}^2
 \leq\frac{\lambda_t^+}{2\eta_t}\|\cdot\|_2^2
 =\frac{L}{2}\frac{\lambda_t^+}{\lambda_t^-}\|\cdot\|_2^2
 \leq\frac{L\kappa}{2}\|\cdot\|_2^2.\label{h-2}
\end{align}
By smoothness of $f$ and~\eqref{2-h},
\begin{align*}
 f(x_{t+1})
 &\leq f(x_t)+\langle\nabla f(x_t),x_{t+1}-x_t\rangle+\frac{L}{2}\|x_{t+1}-x_t\|_2^2\\
 &\leq f(x_t)+\langle\nabla f(x_t),x_{t+1}-x_t\rangle+\frac{1}{2\eta_t}\|x_{t+1}-x_t\|_{H_t}^2\\
 &=f(x_t)+\langle\tilde{\nabla}f(x_t),x_{t+1}-x_t\rangle+\frac{1}{2\eta_t}\|x_{t+1}-x_t\|_{H_t}^2+\langle\nabla f(x_t)-\tilde{\nabla}f(x_t),x_{t+1}-x_t\rangle\\
 &=Q_t(x_{t+1})+\langle\nabla f(x_t)-\tilde{\nabla}f(x_t),x_{t+1}-x_t\rangle\\
 &=Q_t(y_K^{(t)})+\langle\nabla f(x_t)-\tilde{\nabla}f(x_t),x_{t+1}-x_t\rangle\\
 &\leq Q_t(y_0^{(t)}+\gamma_t(v_0^{(t)}-y_0^{(t)}))+\|\nabla f(x_t)-\tilde{\nabla}f(x_t)\|_2\|x_{t+1}-x_t\|_2,
\end{align*}
by Lemma~\ref{lem:sub} and the Cauchy-Schwarz inequality. Recall that $y_0^{(t)}=x_t$ and let $v_t\coloneqq v_0^{(t)}$. Then,
\begin{align}
 f(x_{t+1})
 &\leq Q_t(x_t+\gamma_t(v_t-x_t))+\|\nabla f(x_t)-\tilde{\nabla}f(x_t)\|_2\|x_{t+1}-x_t\|_2\nonumber\\
 &=f(x_t)+\gamma_t\langle\tilde{\nabla}f(x_t),v_t-x_t\rangle+\frac{\gamma_t^2}{2\eta_t}\|v_t-x_t\|_{H_t}^2+\|\nabla f(x_t)-\tilde{\nabla}f(x_t)\|_2\|x_{t+1}-x_t\|_2.\label{proof:1}
\end{align}
Let $x^*\in\argmin_\mathcal{C}f$. Since $\nabla Q_t(y_0^{(t)})=\tilde{\nabla}f(x_t)$, we have $v_t\in\argmin_{v\in\mathcal{C}}\langle\tilde{\nabla}f(x_t),v\rangle$ so
\begin{align}
 \langle\tilde{\nabla}f(x_t),v_t-x_t\rangle
 &\leq\langle\tilde{\nabla}f(x_t),x^*-x_t\rangle\nonumber\\
 &=\langle\nabla f(x_t),x^*-x_t\rangle+\langle\tilde{\nabla}f(x_t)-\nabla f(x_t),x^*-x_t\rangle\nonumber\\
 &\leq f(x^*)-f(x_t)+\|\tilde{\nabla}f(x_t)-\nabla f(x_t)\|_2\|x^*-x_t\|_2,\label{proof:x*}
\end{align}
by convexity of $f$ and the Cauchy-Schwarz inequality. Let $\epsilon_t\coloneqq f(x_t)-\min_\mathcal{C}f$ for all $t\in\llbracket0,T\rrbracket$. Combining~\eqref{proof:1} and~\eqref{proof:x*}, subtracting both sides by $\min_\mathcal{C}f$, and taking the expectation, we obtain
\begin{align}
 \mathsf{E}[\epsilon_{t+1}]
 &\leq(1-\gamma_t)\mathsf{E}[\epsilon_t]+\gamma_t\mathsf{E}[\|\tilde{\nabla}f(x_t)-\nabla f(x_t)\|_2\|x^*-x_t\|_2]+\frac{\gamma_t^2}{2\eta_t}\mathsf{E}[\|v_t-x_t\|_{H_t}^2]\nonumber\\
 &\quad+\mathsf{E}[\|\nabla f(x_t)-\tilde{\nabla}f(x_t)\|_2\|x_{t+1}-x_t\|_2]\nonumber\\
 &\leq(1-\gamma_t)\mathsf{E}[\epsilon_t]+\gamma_t\mathsf{E}[\|\tilde{\nabla}f(x_t)-\nabla f(x_t)\|_2]D+\frac{\gamma_t^2}{2}L\kappa D^2+\mathsf{E}[\|\nabla f(x_t)-\tilde{\nabla}f(x_t)\|_2]KD\gamma_t,\label{proof:sfw}
\end{align}
where we used~\eqref{h-2} and Lemma~\ref{lem:x}. By Lemma~\ref{lem:sfw}, and with $b_t=\big(G(t+2)/(LD)\big)^2$ and $\gamma_t=2/(t+2)$,
\begin{align*}
 \mathsf{E}[\epsilon_{t+1}]
 &\leq(1-\gamma_t)\mathsf{E}[\epsilon_t]+\gamma_t\frac{G}{\sqrt{b_t}}D+\frac{\gamma_t^2}{2}L\kappa D^2+\frac{G}{\sqrt{b_t}}KD\gamma_t\\
 &=\frac{t}{t+2}\mathsf{E}[\epsilon_t]+\frac{2}{t+2}\frac{LD}{t+2}D+\frac{2}{(t+2)^2}L\kappa D^2+\frac{LD}{t+2}\frac{2KD}{t+2}\\
 &=\frac{t}{t+2}\mathsf{E}[\epsilon_t]+\frac{C}{(t+2)^2},
\end{align*}
where $C\coloneqq2LD^2(K+1+\kappa)$. Thus,
\begin{align*}
 (t+1)(t+2)\mathsf{E}[\epsilon_{t+1}]
 \leq t(t+1)\mathsf{E}[\epsilon_t]+\frac{C(t+1)}{t+2},
\end{align*}
so, by telescoping,
\begin{align*}
 t(t+1)\mathsf{E}[\epsilon_t]
 &\leq0\cdot1\cdot\mathsf{E}[\epsilon_0]+\sum_{s=0}^{t-1}\frac{C(s+1)}{s+2}\\
 &\leq Ct
\end{align*}
for all $t\in\llbracket1,T\rrbracket$. Therefore,
\begin{align*}
 \mathsf{E}[\epsilon_t]
 \leq\frac{C}{t+1}
\end{align*}
for all $t\in\llbracket1,T\rrbracket$.
\end{proof}

\begin{theorem}[(Theorem~\ref{th:ncvx})]
 Suppose that $f_1,\ldots,f_m$ are not necessarily convex. Consider AdaSFW (Algorithm~\ref{adasfw}) with $b_t\leftarrow\big(G/(LD)\big)^2(t+1)$, $\eta_t\leftarrow\lambda_t^-/L$, and $\gamma_t\leftarrow1/(t+1)^{1/2+\nu}$ where $\nu\in\left]0,1/2\right[$, and let $\kappa\coloneqq\lambda_0^+/\lambda_0^-$. For all $t\in\llbracket0,T\rrbracket$, let $X_t$ be sampled uniformly at random from $\{x_0,\ldots,x_t\}$. Then for all $t\in\llbracket0,T-1\rrbracket$,
 \begin{align*}
  \mathsf{E}[g(X_t)]
  \leq\frac{(f(x_0)-\min_\mathcal{C}f)+LD^2(K+1+\kappa/2)S}{(t+1)^{1/2-\nu}},
 \end{align*}
 where $S\coloneqq\sum_{s=0}^{+\infty}1/(s+1)^{1+\nu}\in\mathbb{R}_+$. Alternatively, if the time horizon $T$ is fixed, then with $b_t\leftarrow(G/(LD))^2T$ and $\gamma_t\leftarrow1/\sqrt{T}$, 
 \begin{align*}
  \mathsf{E}[g(X_{T-1})]
  \leq\frac{(f(x_0)-\min_\mathcal{C}f)+LD^2(K+1+\kappa/2)}{\sqrt{T}}.
 \end{align*}
\end{theorem}

\begin{proof}
 For all $t\in\llbracket0,T\rrbracket$, let $\mathsf{E}_t$ denote the conditional expectation with respect to the realization of $X_t$ given $\{x_0,\ldots,x_t\}$. Recall that $\mathsf{E}$ denotes the expectation with respect to all the randomness in the system. Let $t\in\llbracket0,T-1\rrbracket$. By~\eqref{proof:1}, 
 \begin{align*}
  f(x_{t+1})
  \leq f(x_t)+\gamma_t\langle\tilde{\nabla}f(x_t),v_t-x_t\rangle+\frac{\gamma_t^2}{2\eta_t}\|v_t-x_t\|_{H_t}^2+\|\nabla f(x_t)-\tilde{\nabla}f(x_t)\|_2\|x_{t+1}-x_t\|_2.
 \end{align*}
 Let $w_t\in\argmin_{v\in\mathcal{C}}\langle\nabla f(x_t),v\rangle$ and note that $g(x_t)=\langle\nabla f(x_t),x_t-w_t\rangle$. Then, since $v_t\in\argmin_{v\in\mathcal{C}}\langle\tilde{\nabla}f(x_t),v\rangle$,
 \begin{align*}
  f(x_{t+1})
  &\leq f(x_t)+\gamma_t\langle\tilde{\nabla}f(x_t),w_t-x_t\rangle+\frac{\gamma_t^2}{2\eta_t}\|v_t-x_t\|_{H_t}^2+\|\nabla f(x_t)-\tilde{\nabla}f(x_t)\|_2\|x_{t+1}-x_t\|_2\\
  &=f(x_t)+\gamma_t\langle\nabla f(x_t),w_t-x_t\rangle+\gamma_t\langle\tilde{\nabla}f(x_t)-\nabla f(x_t),w_t-x_t\rangle\\
  &\quad+\frac{\gamma_t^2}{2\eta_t}\|v_t-x_t\|_{H_t}^2+\|\nabla f(x_t)-\tilde{\nabla}f(x_t)\|_2\|x_{t+1}-x_t\|_2\\
  &\leq f(x_t)-\gamma_tg(x_t)+\gamma_t\|\tilde{\nabla}f(x_t)-\nabla f(x_t)\|_2D+\gamma_t^2\frac{L\kappa}{2}D^2+\|\nabla f(x_t)-\tilde{\nabla}f(x_t)\|_2KD\gamma_t,
 \end{align*}
 where we used the Cauchy-Schwarz inequality, \eqref{h-2}, and Lemma~\ref{lem:x} in the last inequality. By Lemma~\ref{lem:sfw}, we obtain
 \begin{align}
  \mathsf{E}[f(x_{t+1})]
  &\leq\mathsf{E}[f(x_t)]-\gamma_t\mathsf{E}[g(x_t)]+\gamma_t\frac{G}{\sqrt{b_t}}D+\gamma_t^2\frac{L\kappa}{2}D^2+\frac{G}{\sqrt{b_t}}KD\gamma_t\label{nu}\\
  &=\mathsf{E}[f(x_t)]-\gamma_t\mathsf{E}[g(x_t)]+\frac{LD^2}{(t+1)^{1+\nu}}+\frac{L\kappa D^2}{2(t+1)^{1+2\nu}}+\frac{KLD^2}{(t+1)^{1+\nu}}\nonumber\\
  &\leq\mathsf{E}[f(x_t)]-\gamma_t\mathsf{E}[g(x_t)]+\frac{LD^2(K+1+\kappa/2)}{(t+1)^{1+\nu}},\nonumber
 \end{align}
 so, by telescoping,
 \begin{align*}
  \sum_{s=0}^t\gamma_s\mathsf{E}[g(x_s)]
  &\leq\mathsf{E}[f(x_0)]-\mathsf{E}[f(x_{t+1})]+\sum_{s=0}^t\frac{LD^2(K+1+\kappa/2)}{(s+1)^{1+\nu}}\\
  &\leq\left(f(x_0)-\min_\mathcal{C}f\right)+LD^2(K+1+\kappa/2)S
 \end{align*}
 and
 \begin{align*}
  \sum_{s=0}^t\gamma_s\mathsf{E}[g(x_s)]
  &\geq\gamma_t\sum_{s=0}^t\mathsf{E}[g(x_s)]\\
  &=\gamma_t(t+1)\mathsf{E}\!\left[\sum_{s=0}^t\frac{1}{t+1}g(x_s)\right]\\
  &=(t+1)^{1/2-\nu}\mathsf{E}[\mathsf{E}_t[g(X_t)]]\\
  &=(t+1)^{1/2-\nu}\mathsf{E}[g(X_t)],
 \end{align*}
 by the law of total expectation. Therefore,
 \begin{align*}
  \mathsf{E}[g(X_t)]
  \leq\frac{(f(x_0)-\min_\mathcal{C}f)+LD^2(K+1+\kappa/2)S}{(t+1)^{1/2-\nu}}.
 \end{align*}
 At~\eqref{nu}, alternatively,
 \begin{align*}
  \mathsf{E}[f(x_{t+1})]
  &\leq\mathsf{E}[f(x_t)]-\frac{1}{\sqrt{T}}\mathsf{E}[g(x_t)]+\frac{LD^2}{T}+\frac{1}{T}\frac{L\kappa}{2}D^2+\frac{KLD^2}{T}\\
  &=\mathsf{E}[f(x_t)]-\frac{1}{\sqrt{T}}\mathsf{E}[g(x_t)]+\frac{LD^2(K+1+\kappa/2)}{T},
 \end{align*}
 so, by telescoping,
 \begin{align*}
  \frac{1}{\sqrt{T}}\sum_{t=0}^{T-1}\mathsf{E}[g(x_t)]
  &\leq\mathsf{E}[f(x_0)]-\mathsf{E}[f(x_T)]+LD^2(K+1+\kappa/2)\\
  &\leq\left(f(x_0)-\min_\mathcal{C}f\right)+LD^2(K+1+\kappa/2)
 \end{align*}
 and
 \begin{align*}
  \frac{1}{\sqrt{T}}\sum_{t=0}^{T-1}\mathsf{E}[g(x_t)]
  &=\frac{T}{\sqrt{T}}\mathsf{E}\!\left[\sum_{t=0}^{T-1}\frac{1}{T}g(x_t)\right]\\
  &=\sqrt{T}\mathsf{E}[\mathsf{E}_{T-1}[g(X_{T-1})]]\\
  &=\sqrt{T}\mathsf{E}[g(X_{T-1})],
 \end{align*}
 by the law of total expectation. Therefore,
 \begin{align*}
  \mathsf{E}[g(X_{T-1})]
  \leq\frac{(f(x_0)-\min_\mathcal{C}f)+LD^2(K+1+\kappa/2)}{\sqrt{T}}.
 \end{align*}
\end{proof}

\subsection{SVRF with adaptive gradients}

Lemma~\ref{lem:svrf} is a slight modification of \citet[Lemma~1]{hazan16}.

\begin{lemma}
 \label{lem:svrf}
 Consider AdaSVRF (Algorithm~\ref{adasvrf}). For all $t\in\llbracket0,T-1\rrbracket$,
 \begin{align*}
  \mathsf{E}[\|\tilde{\nabla}f(x_t)-\nabla f(x_t)\|_2^2]
  &\leq\frac{4L}{b_t}\left(\mathsf{E}\left[f(x_t)-\min_\mathcal{C}f\right]+\mathsf{E}\left[f(\tilde{x}_t)-\min_\mathcal{C}f\right]\right).
 \end{align*}
\end{lemma}

\begin{proof}
 Let $t\in\llbracket0,T-1\rrbracket$, $\mathsf{E}_t$ denote the conditional expectation with respect to the realization of $i_1,\ldots,i_{b_t}$ given all the randomness in the past (hence, $\tilde{x}_t$ and $x_t$ are given), and $x^*\in\argmin_\mathcal{C}f$. For all $i\in\{i_1,\ldots,i_{b_t}\}$,
 \begin{align*}
  &\mathsf{E}_t[\|\nabla f(x_t)-(\nabla f(\tilde{x}_t)+\nabla f_i(x_t)-\nabla f_i(\tilde{x}_t))\|_2^2]\\
  &=\mathsf{E}_t[\|\nabla f(x_t)-\nabla f(x^*)+\nabla f(x^*)-\nabla f(\tilde{x}_t)-\nabla f_i(x_t)+\nabla f_i(x^*)-\nabla f_i(x^*)+\nabla f_i(\tilde{x}_t)\|_2^2]\\
  &\leq2(\mathsf{E}_t[\|\nabla f(x_t)-\nabla f(x^*)-\nabla f_i(x_t)+\nabla f_i(x^*)\|_2^2]+\mathsf{E}_t[\|\nabla f(x^*)-\nabla f(\tilde{x}_t)-\nabla f_i(x^*)+\nabla f_i(\tilde{x}_t)\|_2^2]),
 \end{align*}
 where we used $(a+b)^2\leq2(a^2+b^2)$ for all $a,b\in\mathbb{R}$. Since $\mathsf{E}_t[\nabla f_i(x)]=\nabla f(x)$ for all $i\in\{i_1,\ldots,i_{b_t}\}$ and $x\in\{x^*,\tilde{x}_t,x_t\}$, then the first term above is the variance of $\nabla f_i(x_t)-\nabla f_i(x^*)$ and the second term above is the variance of $\nabla f_i(\tilde{x}_t)-\nabla f_i(x^*)$, both with respect to $\mathsf{E}_t$. The variance of a random variable being upper bounded by its second moment, we obtain
 \begin{align*}
  &\mathsf{E}_t[\|\nabla f(x_t)-(\nabla f(\tilde{x}_t)+\nabla f_i(x_t)-\nabla f_i(\tilde{x}_t))\|_2^2]\\
  &\leq2(\mathsf{E}_t[\|\nabla f_i(x_t)-\nabla f_i(x^*)\|_2^2]+\mathsf{E}_t[\|\nabla f_i(\tilde{x}_t)-\nabla f_i(x^*)\|_2^2])\\
  &\leq4L(\mathsf{E}_t[f_i(x_t)-f_i(x^*)-\langle\nabla f_i(x^*),x_t-x^*\rangle]+\mathsf{E}_t[f_i(\tilde{x}_t)-f_i(x^*)-\langle\nabla f_i(x^*),\tilde{x}_t-x^*\rangle])\\
  &=4L(\mathsf{E}_t[f_i(x_t)]-\mathsf{E}_t[f_i(x^*)]-\langle\mathsf{E}_t[\nabla f_i(x^*)],x_t-x^*\rangle+\mathsf{E}_t[f_i(\tilde{x}_t)]-\mathsf{E}_t[f_i(x^*)]-\langle\mathsf{E}_t[\nabla f_i(x^*)],\tilde{x}_t-x^*\rangle)\\
  &=4L(f(x_t)-f(x^*)-\langle\nabla f(x^*),x_t-x^*\rangle+f(\tilde{x}_t)-f(x^*)-\langle\nabla f(x^*),\tilde{x}_t-x^*\rangle),
  \end{align*}
  by $L$-smoothness of $f_i$ for all $i\in\{i_1,\ldots,i_{b_t}\}$ and taking the conditional expectation. By convexity of $f$,
  \begin{align*}
  \mathsf{E}_t[\|\nabla f(x_t)-(\nabla f(\tilde{x}_t)+\nabla f_i(x_t)-\nabla f_i(\tilde{x}_t))\|_2^2]
  &\leq4L(f(x_t)-f(x^*)+f(\tilde{x}_t)-f(x^*)).
 \end{align*}
 By the law of total expectation, 
 \begin{align*}
  \mathsf{E}[\|\nabla f(x_t)-\tilde{\nabla}f(x_t)\|_2^2]
  &=\mathsf{E}\!\left[\left\|\frac{1}{b_t}\sum_{i=i_1}^{i_{b_t}}\big(\nabla f(x_t)-(\nabla f(\tilde{x}_t)+\nabla f_i(x_t)-\nabla f_i(\tilde{x}_t))\big)\right\|_2^2\right]\\
  &\leq\frac{1}{b_t^2}\sum_{i=i_1}^{i_{b_t}}\mathsf{E}[\|\nabla f(x_t)-(\nabla f(\tilde{x}_t)+\nabla f_i(x_t)-\nabla f_i(\tilde{x}_t))\|_2^2]\\
  &\leq\frac{4L}{b_t}(\mathsf{E}[f(x_t)-f(x^*)]+\mathsf{E}[f(\tilde{x}_t)-f(x^*)]).
 \end{align*}
\end{proof}

\begin{theorem}[(Theorem~\ref{th:adasvrf})]
 Consider AdaSVRF (Algorithm~\ref{adasvrf}) with $s_k\leftarrow2^k-1$, $b_t\leftarrow24(K+1+\kappa)(t+2)$ where $\kappa\coloneqq\lambda_0^+/\lambda_0^-$, $\eta_t\leftarrow\lambda_t^-/L$, and $\gamma_t\leftarrow2/(t+2)$. Then for all $t\in\llbracket1,T\rrbracket$,
 \begin{align}
  \mathsf{E}[f(x_t)]-\min_\mathcal{C}f
  \leq\frac{2LD^2(K+1+\kappa)}{t+2}.\label{adasvrf:res}
 \end{align}
\end{theorem}

\begin{proof}
We proceed by strong induction. Let $\epsilon_t\coloneqq f(x_t)-\min_\mathcal{C}f$ for all $t\in\llbracket0,T\rrbracket$. By~\eqref{proof:sfw}, 
\begin{align*}
 \mathsf{E}[\epsilon_{t+1}]
 &\leq(1-\gamma_t)\mathsf{E}[\epsilon_t]+\gamma_t\mathsf{E}[\|\tilde{\nabla}f(x_t)-\nabla f(x_t)\|_2]D+\frac{\gamma_t^2}{2}L\kappa D^2+\mathsf{E}[\|\nabla f(x_t)-\tilde{\nabla}f(x_t)\|_2]KD\gamma_t,
\end{align*}
for all $t\in\llbracket0,T-1\rrbracket$. If $t=0$ then, since $s_0=0$, we have $\tilde{\nabla}f(x_0)=\nabla f(x_0)$ so 
\begin{align*}
 \mathsf{E}[\epsilon_1]
 &\leq(1-\gamma_0)\mathsf{E}[\epsilon_0]+\frac{\gamma_0^2}{2}L\kappa D^2\\
 &=\frac{LD^2\kappa}{2},
\end{align*}
because $\gamma_0=1$, so the base case holds. Suppose~\eqref{adasvrf:res} holds for all $t'\in\llbracket1,t\rrbracket$ for some $t\in\llbracket1,T-1\rrbracket$. There exist $k,\ell\in\mathbb{N}$ such that $t=s_k+\ell$ and $\ell\leq s_{k+1}-s_k-1$. That is, $s_k$ is the last snapshot time and $\tilde{x}_t=x_{s_k}$. Note that this implies $\ell\leq2^k-1=s_k$ so $t+2=s_k+\ell+2\leq2s_k+2\leq2(s_k+2)$. By Lemma~\ref{lem:svrf},
\begin{align*}
 \mathsf{E}[\|\tilde{\nabla}f(x_t)-\nabla f(x_t)\|_2^2]
 &\leq\frac{4L}{b_t}\left(\mathsf{E}\left[f(x_t)-\min_\mathcal{C}f\right]+\mathsf{E}\left[f(\tilde{x}_t)-\min_\mathcal{C}f\right]\right)\\
 &\leq\frac{4L}{b_t}\left(\frac{2LD^2(K+1+\kappa)}{t+2}+\frac{2LD^2(K+1+\kappa)}{s_k+2}\right)\\
 &\leq\frac{4L}{b_t}\left(\frac{2LD^2(K+1+\kappa)}{t+2}+\frac{4LD^2(K+1+\kappa)}{t+2}\right)\\
 &=\frac{24L^2D^2(K+1+\kappa)}{b_t(t+2)}\\
 &=\left(\frac{LD}{t+2}\right)^2,
\end{align*}
since $b_t=24(K+1+\kappa)(t+2)$. By Jensen's inequality,
\begin{align*}
 \mathsf{E}[\|\tilde{\nabla}f(x_t)-\nabla f(x_t)\|_2]
 &\leq\sqrt{\mathsf{E}[\|\tilde{\nabla}f(x_t)-\nabla f(x_t)\|_2^2]}\\
 &\leq\frac{LD}{t+2}.
\end{align*}
Thus, with $\gamma_t=2/(t+2)$,
\begin{align*}
 \mathsf{E}[\epsilon_{t+1}]
 &\leq\frac{t}{t+2}\frac{2LD^2(K+1+\kappa)}{t+2}+\frac{2}{t+2}\frac{LD}{t+2}D+\frac{2}{(t+2)^2}L\kappa D^2+\frac{LD}{t+2}KD\frac{2}{t+2}\\
 &=\frac{2LD^2(K+1+\kappa)t+2LD^2+2LD^2\kappa+2KLD^2}{(t+2)^2}\\
 &=\frac{2LD^2(K+1+\kappa)(t+1)}{(t+2)^2}\\
 &\leq\frac{2LD^2(K+1+\kappa)}{t+3}.
\end{align*}
\end{proof}

\subsection{CSFW with adaptive gradients}

Let 
\begin{align*}
 A\coloneqq
 \begin{pmatrix}
  a_1^\top\\
  \vdots\\
  a_m^\top
 \end{pmatrix}
 \in\mathbb{R}^{m\times n}
 \quad\text{and}\quad
 f'(x)\coloneqq\frac{1}{m}\sum_{i=1}^mf_i'(\langle a_i,x\rangle)e_i
\end{align*}
for all $x\in\mathbb{R}^n$, where $(e_1,\ldots,e_m)$ denotes the canonical basis of $\mathbb{R}^m$. Thus,
\begin{align}
 \nabla f(x)=\frac{1}{m}\sum_{i=1}^mf_i'(\langle a_i,x\rangle)a_i=A^\top f'(x)\label{fprime}
\end{align}
for all $x\in\mathbb{R}^n$.

Lemma~\ref{lem:csfw} is adapted from \citet[Lemmata~2 and~3]{negiar20} and uses Lemma~\ref{lem:x}.

\begin{lemma}
 \label{lem:csfw}
 Consider AdaCSFW (Algorithm~\ref{adacsfw}). For all $t\in\llbracket1,T\rrbracket$,
 \begin{align*}
  \mathsf{E}_t[\|f'(x_t)-\alpha_t\|_1]
  \leq\left(1-\frac{b}{m}\right)\left(\|f'(x_{t-1})-\alpha_{t-1}\|_1+\frac{KLD_1^A}{m}\gamma_{t-1}\right).
 \end{align*}
 where $\mathsf{E}_t$ denotes the conditional expectation with respect to the realization of $i_1,\ldots,i_b$ given all the randomness in the past. Thus, for all $t\in\llbracket0,T\rrbracket$,
 \begin{align*}
  \mathsf{E}[\|f'(x_t)-\alpha_t\|_1]
  \leq\left(1-\frac{b}{m}\right)^t\|f'(x_0)-\alpha_0\|_1
  +\frac{2KLD_1^A}{m}\left(\left(1-\frac{b}{m}\right)^{t/2}\ln\left(\frac{t}{2}+1\right)+\frac{2(m/b-1)}{t+2}\right).
 \end{align*}
\end{lemma}

\begin{proof}
 Let $t\in\llbracket1,T\rrbracket$ and $i_1,\ldots,i_b$ be the indices sampled at iteration $t$. For all $i\in\llbracket1,m\rrbracket$, we have
 \begin{align*}
  [\alpha_t]_i=
  \begin{cases}
   (1/m)f_i'(\langle a_i,x_t\rangle)&\text{if }i\in\{i_1,\ldots,i_b\},\\
   [\alpha_{t-1}]_i&\text{if }i\notin\{i_1,\ldots,i_b\}.
  \end{cases}
 \end{align*}
 Thus,
 \begin{align}
  \mathsf{E}_t[\|f'(x_t)-\alpha_t\|_1]
  &=\sum_{i=1}^m\mathsf{E}_t\!\left[\left|\frac{1}{m}f_i'(\langle a_i,x_t\rangle-[\alpha_t]_i\right|\right]\nonumber\\
  &=\sum_{i=1}^m\left(1-\frac{b}{m}\right)\left|\frac{1}{m}f_i'(\langle a_i,x_t\rangle-[\alpha_{t-1}]_i\right|\nonumber\\
  &=\left(1-\frac{b}{m}\right)\|f'(x_t)-\alpha_{t-1}\|_1.\label{lemcsfw:1}
 \end{align}
 Then, by the triangular inequality and $L$-smoothness of $f_1,\ldots,f_m$,
 \begin{align}
  \|f'(x_t)-\alpha_{t-1}\|_1
  &\leq\|f'(x_t)-f'(x_{t-1})\|_1+\|f'(x_{t-1})-\alpha_{t-1}\|_1\nonumber\\
  &=\sum_{i=1}^m\frac{1}{m}|f_i'(\langle a_i,x_t\rangle)-f_i'(\langle a_i,x_{t-1}\rangle)|+\|f'(x_{t-1})-\alpha_{t-1}\|_1\nonumber\\
  &\leq\frac{L}{m}\sum_{i=1}^m|\langle a_i,x_t-x_{t-1}\rangle|+\|f'(x_{t-1})-\alpha_{t-1}\|_1\nonumber\\
  &=\frac{L}{m}\|A(x_t-x_{t-1})\|_1+\|f'(x_{t-1})-\alpha_{t-1}\|_1.\label{lemcsfw:2}
 \end{align}
 Now, similarly to the proof of Lemma~\ref{lem:x},
 \begin{align*}
  \|A(x_t-x_{t-1})\|_1
  &=\left\|\sum_{k=0}^{K-1}\left(\prod_{\ell=k+1}^{K-1}(1-\gamma_\ell^{(t-1)})\right)\gamma_k^{(t-1)}A(v_k^{(t-1)}-x_{t-1})\right\|_1\\
  &\leq\sum_{k=0}^{K-1}1\cdot\gamma_{t-1}\cdot\|A(v_k^{(t-1)}-x_{t-1})\|_1\\
  &\leq K\gamma_{t-1}D_1^A.
 \end{align*}
 Together with~\eqref{lemcsfw:1} and~\eqref{lemcsfw:2}, we obtain
 \begin{align*}
  \mathsf{E}_t[\|f'(x_t)-\alpha_t\|_1]
  \leq\left(1-\frac{b}{m}\right)\left(\|f'(x_{t-1})-\alpha_{t-1}\|_1+\frac{KLD_1^A}{m}\gamma_{t-1}\right).
 \end{align*}
 The second result follows as in \citet[Lemma~3]{negiar20}. 
\end{proof}

\begin{theorem}[(Theorem~\ref{th:adacsfw})]
 Consider AdaCSFW (Algorithm~\ref{adacsfw}) with $\eta_t\leftarrow m\lambda_t^-/(L\|A\|_2^2)$ and $\gamma_t\leftarrow2/(t+2)$, and let $\kappa\coloneqq\lambda_0^+/\lambda_0^-$. Then for all $t\in\llbracket1,T\rrbracket$,
 \begin{align*}
  \mathsf{E}[f(x_t)]-\min_\mathcal{C}f
  &\leq\frac{2L}{t+1}\left(4K(K+1)D_1^AD_\infty^A\left(\frac{1}{b}-\frac{1}{m}\right)+\frac{\kappa\|A\|_2^2D_2^2}{m}\right)\\
  &\quad+\frac{2(K+1)D_\infty^A(m/b)^2}{t(t+1)}\left(\|f'(x_0)-\alpha_0\|_1+\frac{16KLD_1^A}{b}\right).
 \end{align*}
\end{theorem}

\begin{proof}
 By $(L/m)$-smoothness of $\varphi\colon\xi\in\mathbb{R}^m\mapsto(1/m)\sum_{i=1}^mf_i([\xi]_i)$ \citep[Proposition~2]{negiar20}, we have
 \begin{align*}
  \varphi(Ax_{t+1})
  \leq\varphi(Ax_t)+\langle\nabla\varphi(x_t),A(x_{t+1}-x_t)\rangle+\frac{L}{2m}\|A(x_{t+1}-x_t)\|_2^2,
 \end{align*}
 i.e.,
 \begin{align*}
  f(x_{t+1})
  &\leq f(x_t)+\langle\nabla f(x_t),x_{t+1}-x_t\rangle+\frac{L}{2m}\|A(x_{t+1}-x_t)\|_2^2\\
  &\leq f(x_t)+\langle\nabla f(x_t),x_{t+1}-x_t\rangle+\frac{L}{2m}\|A\|_2^2\|x_{t+1}-x_t\|_2^2\\
  &\leq f(x_t)+\langle\nabla f(x_t),x_{t+1}-x_t\rangle+\frac{L}{2m}\frac{\|A\|_2^2}{\lambda_t^-}\|x_{t+1}-x_t\|_{H_t}^2\\
  &=f(x_t)+\langle\nabla f(x_t),x_{t+1}-x_t\rangle+\frac{1}{2\eta_t}\|x_{t+1}-x_t\|_{H_t}^2\\
  &=f(x_t)+\langle\tilde{\nabla}f(x_t),x_{t+1}-x_t\rangle+\frac{1}{2\eta_t}\|x_{t+1}-x_t\|_{H_t}^2+\langle\nabla f(x_t)-\tilde{\nabla}f(x_t),x_{t+1}-x_t\rangle\\
  &=Q_t(x_{t+1})+\langle\nabla f(x_t)-\tilde{\nabla}f(x_t),x_{t+1}-x_t\rangle
 \end{align*}
 by definition of $\|A\|_2$ and~\eqref{h2}. Thus, with $v_t\coloneqq v_0^{(t)}$ and since $x_{t+1}=y_K^{(t)}$, $x_t=y_0^{(t)}$, by Lemma~\ref{lem:sub} we have
 \begin{align*}
  f(x_{t+1})
  &\leq Q_t(x_t+\gamma_t(v_t-x_t))+\langle\nabla f(x_t)-\tilde{\nabla}f(x_t),x_{t+1}-x_t\rangle\\
  &=f(x_t)+\gamma_t\langle\tilde{\nabla}f(x_t),v_t-x_t\rangle+\frac{\gamma_t^2}{2\eta_t}\|v_t-x_t\|_{H_t}^2+\langle f'(x_t)-\alpha_t,A(x_{t+1}-x_t)\rangle
 \end{align*}
 by~\eqref{fprime}. By H\"older's inequality,
 \begin{align*}
  f(x_{t+1})
  &\leq f(x_t)+\gamma_t\langle\tilde{\nabla}f(x_t),x^*-x_t\rangle+\frac{\gamma_t^2}{2\eta_t}\lambda_t^+\|v_t-x_t\|_2^2+\|f'(x_t)-\alpha_t\|_1\|A(x_{t+1}-x_t)\|_\infty\\
  &=f(x_t)+\gamma_t\langle\nabla f(x_t),x^*-x_t\rangle+\gamma_t\langle\tilde{\nabla}f(x_t)-\nabla f(x_t),x^*-x_t\rangle\\
  &\quad+\gamma_t^2\frac{\lambda_t^+}{\lambda_t^-}\frac{L\|A\|_2^2}{2m}\|v_t-x_t\|_2^2+\|f'(x_t)-\alpha_t\|_1\|A(x_{t+1}-x_t)\|_\infty,
 \end{align*}
 so, by Lemma~\ref{lem:csfw},
 \begin{align*}
  \mathsf{E}[\epsilon_{t+1}]
  &\leq(1-\gamma_t)\mathsf{E}[\epsilon_t]+\gamma_t\mathsf{E}[\|f'(x_t)-\alpha_t\|_1]D_\infty^A+\gamma_t^2\frac{\kappa L\|A\|_2^2}{2m}D_2^2+\mathsf{E}[\|f'(x_t)-\alpha_t\|_1]KD_\infty^A\gamma_t\\
  &=(1-\gamma_t)\mathsf{E}[\epsilon_t]+\gamma_t(K+1)D_\infty^A\mathsf{E}[\|f'(x_t)-\alpha_t\|_1]+\gamma_t^2\frac{\kappa L\|A\|_2^2}{2m}D_2^2\\
  &\leq\frac{t}{t+2}\mathsf{E}[\epsilon_t]+\frac{2(K+1)D_\infty^A}{t+2}
  \left(1-\frac{b}{m}\right)^t\|f'(x_0)-\alpha_0\|_1+\\
  &\quad+\frac{4K(K+1)LD_1^AD_\infty^A}{m(t+2)}\left(\left(1-\frac{b}{m}\right)^{t/2}\ln\left(\frac{t}{2}+1\right)+\frac{2(m/b-1)}{t+2}
  \right)+\frac{2\kappa L\|A\|_2^2D_2^2}{m(t+2)^2}.
 \end{align*}
 Thus, multiplying both sides by $(t+1)(t+2)$,
 \begin{align*}
  (t+1)(t+2)\mathsf{E}[\epsilon_{t+1}]
  &\leq t(t+1)\mathsf{E}[\epsilon_t]
  +2(K+1)D_\infty^A(t+1)
  \left(1-\frac{b}{m}\right)^t\|f'(x_0)-\alpha_0\|_1\\
  &\quad+\frac{4K(K+1)LD_1^AD_\infty^A}{m}(t+1)\left(1-\frac{b}{m}\right)^{t/2}\ln\left(\frac{t}{2}+1\right)\\
  &\quad+\frac{8K(K+1)LD_1^AD_\infty^A(m/b-1)+2\kappa L\|A\|_2^2D_2^2}{m}\frac{t+1}{t+2}.
 \end{align*}
 Telescoping, we obtain for all $t\in\llbracket1,T\rrbracket$,
 \begin{align*}
  t(t+1)\mathsf{E}[\epsilon_t]
  &\leq 0\cdot1\cdot\mathsf{E}[\epsilon_0]
  +2(K+1)D_\infty^A
  \|f'(x_0)-\alpha_0\|_1\sum_{s=0}^{t-1}(s+1)\left(1-\frac{b}{m}\right)^s\\
  &\quad+\frac{4K(K+1)LD_1^AD_\infty^A}{m}\sum_{s=0}^{t-1}(s+1)\left(1-\frac{b}{m}\right)^{s/2}\ln\left(\frac{s}{2}+1\right)\\
  &\quad+\frac{2L}{m}\left(4K(K+1)D_1^AD_\infty^A\left(\frac{m}{b}-1\right)+\kappa\|A\|_2^2D_2^2\right)\sum_{s=0}^{t-1}\frac{s+1}{s+2}\\
  &\leq2(K+1)D_\infty^A
  \|f'(x_0)-\alpha_0\|_1\cdot\left(\frac{m}{b}\right)^2\\
  &\quad+\frac{4K(K+1)LD_1^AD_\infty^A}{m}\cdot8\left(\frac{m}{b}\right)^3\\
  &\quad+\frac{2L}{m}\left(4K(K+1)D_1^AD_\infty^A\left(\frac{m}{b}-1\right)+\kappa\|A\|_2^2D_2^2\right)t.
 \end{align*}
 Therefore,
 \begin{align*}
  \mathsf{E}[\epsilon_t]
  &\leq\frac{2(K+1)D_\infty^A
  \|f'(x_0)-\alpha_0\|_1(m/b)^2}{t(t+1)}\\
  &\quad+\frac{32K(K+1)LD_1^AD_\infty^A(m/b)^2}{bt(t+1)}\\
  &\quad+\frac{2L}{m}\frac{4K(K+1)D_1^AD_\infty^A(m/b-1)+\kappa\|A\|_2^2D_2^2}{t+1}.
 \end{align*}
\end{proof}

\end{document}